\documentclass[a4paper]{article}

\usepackage{natbib}

\usepackage[all]{xy}\usepackage[latin1]{inputenc}       
\usepackage[dvips]{graphics,graphicx}
\usepackage{amsfonts,amssymb,amsmath,xcolor,mathrsfs, amstext}
\usepackage{amsbsy, amsopn, amscd, amsxtra, amsthm, authblk, enumerate,dsfont}
\usepackage{upref}
\usepackage{geometry}
\usepackage[displaymath]{lineno}
\usepackage{float}
\usepackage{bbm}
\usepackage{multicol}
\usepackage{comment}
\usepackage{longtable, booktabs}

\allowdisplaybreaks

\usepackage[colorlinks,
            linkcolor=blue,
            anchorcolor=green,
            citecolor=blue
            ]{hyperref}

\numberwithin{equation}{section}

\newtheorem{theorem}{Theorem}[section]
\newtheorem{lemma}[theorem]{Lemma}
\newtheorem{remark}[theorem]{Remark}

\newtheorem{proposition}[theorem]{Proposition}
\newtheorem*{proposition*}{Proposition}

\newtheorem*{definitions*}{Definitions}

\newtheorem*{example*}{\bf Example}

\numberwithin{equation}{section}

\newcommand\dd{\mathop{}\!\mathrm{d}} 

\DeclareMathOperator{\supp}{supp}
\DeclareMathOperator{\dist}{dist}

\newcommand{\eqb}{\begin{equation}}
\newcommand{\eqe}{\end{equation}}

\title{Conformal non-removability of the Brownian graph}
\author{Haoyu Liu\thanks{Peking University}}

\begin{document}

\maketitle

\begin{abstract}
We prove that the graph of a standard one-dimensional Brownian motion, $\Gamma_B=\{(t,B_t):t \in [0,1]\}$, is almost surely not conformally removable. This resolves the question left open by Doherty and Miller in their study of the Sobolev removability of the Brownian graph~\cite{DM26}. We also construct non-removable $\alpha$-H\"older graphs for every $\alpha \in (0,2/3)$, including the previously open range $\alpha \in [1/2, 2/3)$. Together with the known removability result for $\alpha>2/3$, this identifies $2/3$ as the threshold for conformal removability of H\"older graphs.
\end{abstract}

\setcounter{tocdepth}{1}
\tableofcontents

\section{Introduction}\label{sec:intro}

A compact set $E \subset \widehat{\mathbb C}$ is called \emph{conformally removable} if every homeomorphism of $\widehat{\mathbb C} $ that is conformal on $\widehat{\mathbb C} \setminus E$ is in fact conformal on $\widehat{\mathbb C}$, and thus M\"obius. Conformal removability has been extensively studied in complex analysis and fractal geometry. It plays an important role, for instance, in the uniqueness theory of conformal welding: removability of the welding interface implies uniqueness of the resulting conformal structure up to post-composition by a M\"obius transformation. We refer to~\cite[Section 5]{Younsi15} for a comprehensive survey.

A closely related concept is given by removability of continuous Sobolev functions. For $1 \le p \le \infty$, a compact set $E \subset \mathbb C$ is called \emph{$W^{1,p}$-removable} if every continuous function on a neighborhood $D$ of $E$ that belongs to $W^{1,p}(D \setminus E)$ also belongs to $W^{1,p}(D)$. This property is independent of the choice of $D$, so the neighborhood may be omitted from the terminology. For $1 \le p<q \le \infty$, $W^{1,p}$-removability implies $W^{1,q}$-removability. Moreover, it is known that $W^{1,2}$-removability implies conformal removability, while the converse is a longstanding conjecture~\cite{Jones1995,JS00,Nta-CNED}. For $p>2$, however, Sobolev removability does not in general imply conformal removability, as the example of the Sierpi\'nski gasket below illustrates.

Several classical sufficient conditions for conformal removability are known. For example, a compact set that is a countable union of rectifiable curves is removable. The Jones--Smirnov criterion implies, in particular, that boundaries of H\"older domains are removable~\cite{JS00}. In the opposite direction, it is known that compact sets with nonempty interior are conformally non-removable, whereas fractals with empty interior exhibit subtler behavior. Let $\mathtt C \subset [0,1]$ be the standard $1/3$-Cantor set. Then $\mathtt C \times [0,1]$ is conformally non-removable; see, e.g.,~\cite{Nta-Sobolev}. The standard Sierpi\'nski carpet is also conformally non-removable, since it contains $\mathtt C \times [0,1]$. The case of the Sierpi\'nski gasket is considerably more delicate, and it was proved recently by Ntalampekos~\cite{Nta-gasket,Nta-Sobolev} that it is conformally non-removable and is $W^{1,p}$-removable if and only if $p \in (2,\infty]$. In general, there is no simple characterization of conformally removable sets. In particular, it was proved in~\cite{bishop2020conformal} that the collection of conformally removable compact subsets of $[0,1]^2$ is not Borel with respect to the Hausdorff metric.

Random planar curves provide another natural setting for studying conformal removability. One notable example is the Schramm--Loewner evolution (SLE), whose removability is closely connected to the conformal welding of Liouville quantum gravity surfaces~\cite{Sheffield-zipper,DMS21-LQG-MRT}. For $\kappa \in (0,4)$, an SLE$_\kappa$ curve is the boundary of a H\"older domain and is therefore conformally removable~\cite{RS05,JS00}. Removability also holds in the critical case $\kappa=4$~\cite{KMS-SLE-4}. In the non-simple regime $\kappa \in (4,8)$, Kavvadias, Miller, and Schoug proved removability whenever the adjacency graph of the complementary connected components is connected~\cite{KMS-nonsimple}; this condition is shown to hold for $\kappa \in (4,\kappa_0]$, where $\kappa_0 \approx 5.6158$~\cite{GP20}. In joint work with Zhuang~\cite{LZ26}, the author proved that SLE$_\kappa$ contains a topological Sierpi\'nski carpet when $\kappa$ is sufficiently close to 8. Combined with~\cite{Nta-carpet}, this yields conformal non-removability, as well as non-uniqueness of conformal welding, in that regime. Conformal removability remains open for the other values of $\kappa$ in $(4,8)$.

The special case of graphs of continuous functions has also been studied extensively. Recall that, for $\alpha \in (0,1]$, a function $f:[0,1] \to \mathbb R$ is \emph{$\alpha$-H\"older continuous} if there exists $C>0$ such that $|f(x)-f(y)| \le C|x-y|^\alpha$ for every $x,y \in [0,1]$. The graph $\Gamma_f:=\{(t,f(t)) : t \in [0,1]\}$ is called an $\alpha$-H\"older graph. Kaufman constructed conformally non-removable graphs using Cauchy transforms of measures supported on the graph~\cite{Kaufman1984,Kaufman1986}; his constructions include $\alpha$-H\"older graphs for every $\alpha<1/2$. Tecu developed a systematic rectangular ``Devil's staircase'' construction for Sobolev non-removability~\cite{Tecu10}. For $1<p<\infty$, he proved that every $\alpha$-H\"older graph is $W^{1,p}$-removable when $\alpha>p/(2p-1)$, and constructed non-removable examples for $\alpha<p/(2p-1)$. At $p=\infty$, every $\alpha$-H\"older graph with $\alpha>1/2$ is $W^{1,\infty}$-removable, whereas the constructions of Kaufman and Tecu give graphs with $\alpha<1/2$ that are neither $W^{1,\infty}$-removable nor conformally removable. In a separate note~\cite{Holder-manuscript}, removability is established at the critical exponent $\alpha=p/(2p-1)$ for every $1 \le p<\infty$, as well as at $\alpha=1/2$ for $p=\infty$. In particular, every $2/3$-H\"older graph is $W^{1,2}$-removable and hence conformally removable. Thus every $\alpha$-H\"older graph with $\alpha \ge 2/3$ is conformally removable. Prior to the present work, conformal removability of H\"older graphs was unresolved throughout the range $1/2 \le \alpha<2/3$.

Our first main result concerns the graph of a standard one-dimensional Brownian motion,
\[ \Gamma_B=\{(t,B_t):0\le t \le 1\} \subseteq \mathbb C. \]
Almost surely, the Brownian motion is $\alpha$-H\"older continuous for every $\alpha<1/2$, but not for $\alpha=1/2$, and $\Gamma_B$ has Hausdorff dimension $3/2$; see, for example,~\cite{MP10}. The complex analytic properties of the Brownian graph have recently attracted attention. Binder, Hakobyan, and Li proved that the Brownian graph is almost surely minimal for conformal dimension~\cite{BHL25}. More recently, Doherty and Miller completely characterized its Sobolev removability: almost surely it is not $W^{1,p}$-removable for any finite $p$, while it is $W^{1,\infty}$-removable~\cite{DM26}. Their non-removability proof is a stochastic implementation of Tecu's rectangular construction. They left conformal removability open and observed that the direct mechanism in the earlier constructions of Kaufman and Tecu cannot settle the question, since the resulting exceptional maps already violate $W^{1,\infty}$-removability.

The main purpose of this paper is to resolve this question.

\begin{theorem}\label{thm:main}
Let $B$ be a standard one-dimensional Brownian motion. Then $\Gamma_B$ is almost surely not conformally removable.
\end{theorem}

Our proof gives the following more precise statement.

\begin{theorem}\label{thm:construction}
Almost surely, there exist a nonzero finite positive Radon measure $\mu$ supported on $\Gamma_B$ and a continuous function $f:\mathbb C \to \mathbb C$ with $\bar\partial f=\mu$ such that, for every sufficiently small $\epsilon>0$, the map $z \mapsto z+\epsilon f(z)$ extends to a non-M\"obius homeomorphism of $\widehat{\mathbb C}$ that is conformal on $\widehat{\mathbb C} \setminus \supp \mu$.
\end{theorem}

Since $\supp \mu \subseteq \Gamma_B$, Theorem~\ref{thm:main} follows immediately from Theorem~\ref{thm:construction}. In fact, the function $f$ constructed in the proof is a continuous representative of the Cauchy transform of $\mu$; see Remark~\ref{rem:identification}.

The same analytic criterion also yields a deterministic construction of conformally non-removable H\"older graphs.

\begin{theorem}\label{thm:holder-graph}
    For every $0<\alpha<2/3$, there exists an $\alpha$-H\"older continuous function $f:[0,1] \to \mathbb R$ such that its graph $\Gamma_f=\{(t,f(t)):0 \le t \le 1\}$ is not conformally removable.
\end{theorem}

In particular, this settles the previously open range $1/2 \le \alpha<2/3$. Combined with the removability result for $\alpha \ge 2/3$~\cite{Tecu10,Holder-manuscript}, Theorem~\ref{thm:holder-graph} identifies $2/3$ as the sharp threshold for conformal removability of H\"older graphs, matching the $W^{1,2}$-removability threshold.

For $1/2 \le \alpha<2/3$, the graphs in Theorem~\ref{thm:holder-graph} are $W^{1,\infty}$-removable by~\cite{Tecu10,Holder-manuscript}, whereas Theorem~\ref{thm:main} and~\cite{DM26} give the same separation for the Brownian graph. To the best of our knowledge, these are the first examples of graphs that are $W^{1,\infty}$-removable but conformally non-removable.

Moreover, Theorem~\ref{thm:main} and~\cite{DM26} show that the Brownian graph is almost surely non-removable in both the $W^{1,2}$ and conformal senses. This, together with the agreement of the H\"older thresholds, provides further evidence for the conjectured equivalence between $W^{1,2}$-removability and conformal removability.

Let us now give an overview of the proofs. The proofs of Theorems~\ref{thm:main} and~\ref{thm:holder-graph} build on the mass-redistribution method of Tecu~\cite{Tecu10} and, in the Brownian case, its stochastic implementation by Doherty and Miller~\cite{DM26}. For Sobolev non-removability, the aim is to construct a continuous function with a singular distributional derivative supported on the graph and suitably integrable derivatives off it. To prove conformal non-removability, however, we must additionally obtain a complex-valued function that is holomorphic off the graph and gives a globally injective perturbation of the identity. To achieve this, we use Kaufman's idea of constructing conformal perturbations from Cauchy transforms~\cite{Kaufman1986}.

The common analytic criterion is given in Propositions~\ref{prop:criterion} and~\ref{prop:criterion+}. Following Kaufman, we use the positivity of the measure in the Cauchy transform to establish injectivity. The important difference is that we do not require a bound on the full complex derivative off the graph. More precisely, Proposition~\ref{prop:criterion} applies to a continuous function $f$ satisfying $\bar\partial f=\mu$ and $\partial f=\mu+g \dd A$, where $\mu$ is a nonzero finite positive measure with compact support and $g$ is locally integrable with bounded real part. The imaginary part of $g$ cancels when taking the symmetric part of the distributional derivative matrix, while $\mu$ contributes a nonnegative term. This ensures that, for every sufficiently small $\epsilon>0$, the map $z\mapsto z+\epsilon f(z)$ is strongly monotone in the sense of~\eqref{eq:strong-monotonicity}, and hence extends to a non-M\"obius homeomorphism of the sphere that is conformal off $\supp \mu$.

Proposition~\ref{prop:criterion+} puts this criterion into a form suitable for our constructions. We seek a continuous, compactly supported, real-valued function $v$ whose horizontal distributional derivative is the sum of a positive measure $\mu$ on the graph and a smooth density, while its vertical derivative has a compactly supported density $b \in L^q$ for some $q>2$. The resulting measure $\mu$ and a continuous representative $f$ of its Cauchy transform (see Remark~\ref{rem:identification}) are natural candidates for applying Proposition~\ref{prop:criterion}. The remaining analytic condition then reduces to boundedness of the mixed Riesz transform $\mathcal R_x\mathcal R_y b$.

For the Brownian graph, we first construct a regular multi-scale tree of rectangles associated with upcrossings at a geometric sequence of heights (Proposition~\ref{prop:construct-tree}). The construction is similar to that of~\cite{DM26}, except that we additionally require each retained upcrossing to have duration bounded by a fixed constant times the square of its height. We also establish a separate uniform upper bound on the number of crossings meeting a time interval (Lemma~\ref{lem:uniform-count}). These additional ingredients are the probabilistic input needed for the singular-integral estimate later in the proof.

On this tree, we carry out a smooth version of the mass redistribution, following the same general approach as Tecu and Doherty--Miller. At each generation, the mass attached to a parent box is divided equally among the selected descendants within each row, while the contributions from adjacent rows form a smooth partition of unity. Lemma~\ref{lem:measure-construction} shows that the resulting functions converge to a continuous real-valued function whose horizontal derivative is a nonzero positive measure supported on the Brownian graph. After subtracting an initial smooth term, we obtain a function whose vertical derivative has a compactly supported $L^q$ density $b$ for some $q>2$, as shown in Lemma~\ref{lem:vertical-derivative}.

The main additional estimate is the boundedness of $\mathcal R_x\mathcal R_y b$, proved in Proposition~\ref{prop:riesz-bounded}. Each contribution to $b$ is a product of a bounded horizontal function $G_{P,j}$ and a smooth vertical derivative. The oddness of the mixed Riesz kernel $K_{\mathcal T}$ in the vertical variable gives a pointwise estimate for such a product that does not involve derivatives of $G_{P,j}$ (Lemma~\ref{lem:single}). The estimate is proportional to the length of the supporting time interval and decays away from the corresponding rectangle. The duration restriction controls the former, while Lemma~\ref{lem:uniform-count} controls how many rectangles can contribute near a given point. Summing over dyadic annuli and then over generations gives uniform convergence of the series of mixed Riesz transforms. An application of Proposition~\ref{prop:criterion+} then proves Theorem~\ref{thm:construction}, and hence Theorem~\ref{thm:main}.

For the H\"older graphs, the strategy is similar, but we must construct the graph itself and sharpen the singular-integral estimate to reach higher H\"older exponents (recall that the Brownian graph is not $1/2$-H\"older continuous). We construct the graph by successive piecewise linear replacements and retain intervals corresponding to both upcrossings and downcrossings. To obtain the mixed Riesz transform bound, however, we impose an additional condition on the retained intervals: in each row, their average midpoint must coincide with the midpoint of the parent interval (Lemma~\ref{lem:holder-center}). Choosing the horizontal probability densities to be symmetric about these midpoints then ensures that every horizontal redistribution term $G_{P,j}$ has integral zero (Lemma~\ref{lem:G-p-j-sum-zero}). This cancellation replaces the use of the Brownian counting estimate and yields boundedness of the mixed Riesz transform (Proposition~\ref{prop:riesz-bounded-holder}), thereby proving Theorem~\ref{thm:holder-graph}.

\medskip
\noindent\textbf{Organization of the paper.}
Section~\ref{sec:analysis-prelim} collects the analytic preliminaries and proves the criteria for conformal non-removability. Section~\ref{sec:prob-input} establishes the probabilistic estimates and constructs the regular tree of Brownian upcrossing rectangles with the required duration bound. In Section~\ref{sec:construct-measure}, we use this tree to construct the singular measure and the vertical derivative, prove the mixed Riesz transform bound, and conclude the proofs of Theorems~\ref{thm:construction} and~\ref{thm:main}. Finally, Section~\ref{sec:holder-graph} gives a deterministic construction of non-removable H\"older graphs and proves Theorem~\ref{thm:holder-graph}.

\medskip
\noindent\textbf{Basic Notations.}
We identify $\mathbb C$ with $\mathbb R^2$, write $z=x+iy$, and denote Lebesgue measure by $\dd A=\dd x \dd y$. The Riemann sphere is $\widehat{\mathbb C}=\mathbb C \cup \{\infty\}$. The symbols $\cdot$ and $|\cdot|$ denote the Euclidean inner product and norm on $\mathbb R^2$, respectively. For $z \in \mathbb C$ and $r>0$, $B(z,r)=\{w \in \mathbb C : |w-z|<r\}$ is the open Euclidean disk. For $A \subseteq \mathbb C$, we set $\mathrm{diam}(A)=\sup\{|x-y| : x,y \in A\}$ and $\dist(z,A)=\inf\{ |z-x| : x \in A\}$. The length of an interval $I \subset \mathbb R$ is written as $\mathrm{len}(I)$. For a measure $\sigma$, its closed support is denoted by $\supp \sigma$. For $p \in [1,\infty]$, we use $\|\cdot\|_{L^p}$ for the $L^p$-norm of a function, and we write $L_c^p(\mathbb C)$ for the functions in $L^p(\mathbb C)$ with compact essential support. For an open set $\Omega$, $C_c(\Omega)$ and $C_c^\infty(\Omega)$ denote the spaces of continuous and smooth functions with compactly support in $\Omega$. Finally, $X \lesssim Y$ means that $X\le CY$ for some finite constant $C$ depending only on the fixed parameters that will be clear from the context.

\medskip
\noindent\textbf{Acknowledgement.} We thank Zijie Zhuang for helpful discussions at an earlier stage of this project. AI tools were used in a supporting role in the preparation and editing of the manuscript. The author takes full responsibility for its contents. The author is partially supported by National Key R\&D Program of China (No.\ 2023YFA1010700) and the Fundamental Research Funds for the Central Universities, Peking University.

\section{Analytic preliminaries}\label{sec:analysis-prelim}

In this section, we collect the deterministic analytic tools used in the Brownian and H\"older constructions. We review some basic concepts in distributions and measures in Section~\ref{subsec:prelim-measure}. Section~\ref{subsec:transforms-prelim} then records the relevant identities for the Riesz, Cauchy, and Beurling transforms. In Section~\ref{subsec:criterion}, we prove criteria for conformal non-removability: Proposition~\ref{prop:criterion} gives a useful condition ensuring global injectivity of a conformal perturbation of the identity, and Proposition~\ref{prop:criterion+} puts this condition into the form used in Sections~\ref{sec:construct-measure} and~\ref{sec:holder-graph}.

\subsection{Distributional derivatives and measures}\label{subsec:prelim-measure}

We first recall some notions in distributions. For any domain $\Omega \subseteq \mathbb C$, let $C_c^\infty(\Omega)$ denote the space of compactly supported smooth test functions, and let $\mathcal D'(\Omega)$ be its dual space, i.e., the space of distributions on $\Omega$.

For $T \in \mathcal D'(\mathbb C)$ and $\psi \in C_c^\infty(\mathbb C)$, we write $\langle T,\psi\rangle$ for the action of the distribution $T$ on the test function $\psi$. If $u \in L^1_{\mathrm{loc}}(\mathbb C)$, then it defines a regular distribution $T_u \in \mathcal D'(\mathbb C)$ satisfying $\langle T_u,\psi \rangle=\int_{\mathbb C} u \psi \dd A$ for every $\psi\in C_c^\infty(\mathbb C)$. We then identify $u$ with $T_u$, and write $\langle u,\psi \rangle$ instead of $\langle T_u,\psi \rangle$. The distributional derivatives of $u$ are defined by
\[ \langle D_xu,\psi\rangle:=-\int_{\mathbb C} u\,\partial_x\psi \dd A \qquad \mbox{and} \qquad \langle D_yu,\psi\rangle:=-\int_{\mathbb C} u\,\partial_y\psi \dd A. \]
Similarly, a signed or complex Radon measure $\mu$ defines a distribution $T_\mu \in \mathcal D'(\mathbb C)$ satisfying $\langle \mu,\psi\rangle= \int_{\mathbb C} \psi \dd \mu$, and we identify $\mu$ with $T_\mu$. Such a measure is called finite if its total variation on $\mathbb C$ is finite.

If $f=u+iv\in L^1_{\mathrm{loc}}(\mathbb C;\mathbb C)$, with $u$ and $v$ real-valued, we may view $f$ as the map $(u,v):\mathbb R^2\to\mathbb R^2$ and write
\[ Df=
    \begin{pmatrix}
        D_xu & D_yu \\
        D_xv & D_yv
    \end{pmatrix} \]
for its distributional Jacobian. A real $2\times2$ matrix-valued Radon measure $M=(M_{jk})$ is a matrix whose entries are signed Radon measures. Let $I$ be the $2 \times 2$ identity matrix. For $c \in \mathbb R$, we write the inequality $M \ge cI \dd A$ if $v^TMv \ge c|v|^2 \dd A$ as signed Radon measures for every $v \in \mathbb R^2$; that is, for every $v \in\mathbb R^2$ and every nonnegative $\psi \in C_c^\infty(\mathbb C)$,
\[ \langle v^TMv,\psi\rangle \ge c|v|^2\int_{\mathbb C}\psi \dd A. \]
We also write $\operatorname{sym} M=\frac{1}{2}(M+M^T)$ for the symmetric part of $M$. Note that since $v^TMv=v^T(\operatorname{sym} M)v$, the preceding inequality depends only on the symmetric part of $M$.

We will also use the standard notation of Wirtinger derivatives, $\bar\partial=\partial_{\bar{z}}=\frac{1}{2}(\partial_x+i\partial_y)$ and $\partial=\partial_z=\frac{1}{2}(\partial_x-i\partial_y)$, which act on functions or distributions. In particular, for $f=u+iv$ we have, distributionally,
\begin{equation}\label{eq:distributional-derivatives}
    \bar\partial f=\frac{1}{2}(D_xu-D_yv)+\frac{i}{2}(D_xv+D_yu) \quad \mbox{and} \quad \partial f=\frac{1}{2}(D_xu+D_yv)+\frac{i}{2}(D_xv-D_yu).
\end{equation}

\subsection{Riesz, Cauchy, and Beurling transforms}
\label{subsec:transforms-prelim}

Let us now review some singular integral operators and their basic properties. We refer readers to~\cite{Stein70,AIM09} for more details.

Let $\mathcal S(\mathbb R^2)$ denote the Schwartz space on $\mathbb R^2$. For $h \in\mathcal S(\mathbb R^2)$, we use the following convention for Fourier transform
\[ \widehat h(\xi,\eta)=\int_{\mathbb R^2} h(x,y)e^{-i(x\xi+y\eta)} \dd x \dd y, \]
and the inverse Fourier transform is given by
\[ h(x,y)=\frac{1}{(2\pi)^2} \int_{\mathbb R^2} \widehat h(\xi,\eta)e^{i(x\xi+y\eta)} \dd \xi \dd \eta. \]

For $h \in \mathcal S(\mathbb R^2)$, the planar Riesz transforms $\mathcal R_x$ and $\mathcal R_y$ are the Fourier multiplier operators such that
\[ \widehat{\mathcal R_x h}(\xi,\eta)=-\frac{i\xi}{(\xi^2+\eta^2)^{1/2}}\widehat h(\xi,\eta) \qquad \mbox{and} \qquad \widehat{\mathcal R_y h}(\xi,\eta)=-\frac{i\eta}{(\xi^2+\eta^2)^{1/2}}\widehat h(\xi,\eta) \]
for $(\xi,\eta) \neq (0,0)$. Equivalently,
\[ \mathcal R_x h(z):=\operatorname{p.v.} \int_{\mathbb C} K_x(z-w)h(w) \dd A(w) \quad \mbox{and} \quad \mathcal R_y h(z):=\operatorname{p.v.} \int_{\mathbb C} K_y(z-w)h(w) \dd A(w), \]
where $K_x(X,Y)=\frac{1}{2\pi}\frac{X}{(X^2+Y^2)^{3/2}}$ and $K_y(X,Y)=\frac{1}{2\pi}\frac{Y}{(X^2+Y^2)^{3/2}}$. Here and below, the principal value integral on $\mathbb C$ shall be understood as the limit of the integral as $\delta \downarrow 0$ when restricting the domain to $\{w \in \mathbb C:|w-z|>\delta\}$.

The transforms $\mathcal R_x$ and $\mathcal R_y$ commute on $\mathcal S$, and their composition $\mathcal T:=\mathcal R_x \mathcal R_y=\mathcal R_y \mathcal R_x$ has Fourier multiplier $\widehat{\mathcal T h}(\xi,\eta)=-\frac{\xi\eta}{\xi^2+\eta^2} \widehat h(\xi,\eta)$. Since $\mathcal T=\partial_x \partial_y (-\Delta)^{-1}$ and the fundamental solution of $-\Delta$ in the plane is $-(2\pi)^{-1} \log |z|$, the principal-value kernel of $\mathcal T$ is
\begin{equation}\label{eq:mixed-riesz-kernel-prelim}
    K_{\mathcal T}(X,Y)=\frac{1}{\pi}\frac{XY}{(X^2+Y^2)^2},
\end{equation}
and thus
\[ \mathcal T h(z)=\operatorname{p.v.} \int_{\mathbb C} K_{\mathcal T}(z-w)h(w) \dd A(w). \]

For a compactly supported finite complex Radon measure $\sigma$, its Cauchy transform is the locally integrable distribution $\mathcal C\sigma=E*\sigma$, where $E(z)=1/(\pi z)$. The local integrability follows from $1/|z| \in L^1_{\rm loc}(\mathbb C)$ and Fubini's theorem. At every point where the integral is absolutely convergent,
\[ \mathcal C\sigma(z)=\frac{1}{\pi} \int_{\mathbb C}\frac{1}{z-w} \dd \sigma(w). \]
Since $\bar\partial E=\delta_0$, we have $\bar\partial \mathcal C \sigma=\sigma$.

For $h \in \mathcal S(\mathbb C)$, the Beurling transform is defined as
\[ \mathcal Bh(z)=-\frac1\pi\operatorname{p.v.} \int_{\mathbb C}\frac{h(w)}{(z-w)^2} \dd A(w). \]
Since $\partial E=-1/(\pi z^2)$ in the principal-value sense, for $q>2$ and $h \in L_c^q(\mathbb C)$, we have $\bar\partial \mathcal C(h \dd A)=h \dd A$ and $\partial \mathcal C(h \dd A)=(\mathcal B h) \dd A$; see, e.g.,~\cite[Theorem 4.3.10]{AIM09}.

For $(\xi,\eta) \neq (0,0)$, the Fourier multiplier of $\mathcal B$ is $\frac{\xi-i\eta}{\xi+i\eta}=\frac{\xi^2-\eta^2}{\xi^2+\eta^2}-2i\frac{\xi\eta}{\xi^2+\eta^2}$. Consequently,
\begin{equation}\label{eq:beurling-riesz}
    \mathcal B=\mathcal R_y^2-\mathcal R_x^2+2i \mathcal R_x \mathcal R_y \qquad \mbox{and} \qquad \mathcal R_x^2+\mathcal R_y^2=-\operatorname{id}.
\end{equation}

We shall use the following standard facts. For every $1<p<\infty$, the Riesz transforms $\mathcal R_x$ and $\mathcal R_y$ extend to bounded operators on $L^p(\mathbb C)$, and consequently so do all their finite compositions; see~\cite[Chapter II, Theorem 3, and Chapter III, Section 1]{Stein70}. The Beurling transform $\mathcal B$ also extends to a bounded operator on $L^p(\mathbb C)$ for every $1<p<\infty$; see~\cite[Theorem 4.5.3]{AIM09}.

Moreover, $\mathcal R_x$, $\mathcal R_y$, their finite compositions, and $\mathcal B$ map $C_c^\infty(\mathbb C)$ into $C^\infty(\mathbb C)\cap L^\infty(\mathbb C)$. Indeed, these operators commute with distributional derivatives, and their $L^p$-boundedness implies that, for every $\psi \in C_c^\infty(\mathbb C)$, the image of $\psi$ under any one of these operators belongs to $W^{k,p}(\mathbb C)$ for every $k \ge 0$ and every $1<p<\infty$. The asserted smoothness and boundedness then follow from Sobolev embedding.

Finally, if $q>2$ and $h \in L_c^q(\mathbb C)$, then the Cauchy transform $\mathcal C(h \dd A)$ has a H\"older continuous representative of exponent $1-2/q$; see~\cite[Theorem 4.3.13]{AIM09}.

\subsection{Criteria for conformal non-removability}
\label{subsec:criterion}

We now isolate the analytic mechanism underlying both non-removability constructions. The argument is inspired by Kaufman's construction~\cite{Kaufman1986}.

\begin{proposition}\label{prop:criterion}
Let $f:\mathbb C \to \mathbb C$ be continuous, and suppose that, in the sense of distributions,
\begin{equation}\label{eq:criterion-derivatives}
    \bar\partial f=\nu, \qquad \partial f=\nu+g \dd A,
\end{equation}
where $\nu$ is a nonzero finite positive measure with compact support, $g \in L_{\rm loc}^1(\mathbb C)$ and $\operatorname{Re} g \in L^\infty(\mathbb C)$. For any $\epsilon>0$ such that $\epsilon\|\operatorname{Re} g\|_{L^\infty(\mathbb C)}<1$, write $c_0:=1-\epsilon\|\operatorname{Re} g\|_{L^\infty(\mathbb C)}>0$.
Then $F(z):=z+\epsilon f(z)$ is a homeomorphism of $\mathbb C$ onto itself and satisfies
\begin{equation}\label{eq:strong-monotonicity}
    (F(z)-F(w)) \cdot (z-w) \ge c_0|z-w|^2 \qquad \mbox{for } z,w\in\mathbb C.
\end{equation}
Moreover, $F$ extends to a homeomorphism of $\widehat{\mathbb C}$, is conformal at every point of
$\widehat{\mathbb C}\setminus\supp\nu$, and is not a
M\"obius transformation.
\end{proposition}

\begin{proof}
Write $f=u+iv$ and $g=\alpha+i\beta$, where $\alpha=\operatorname{Re} g$. By the identities for distributional derivatives~\eqref{eq:distributional-derivatives}, the conditions~\eqref{eq:criterion-derivatives} are equivalent to
\[ D_xu=2\nu+\alpha \dd A, \qquad D_yu=-\beta \dd A, \qquad D_xv=\beta \dd A, \qquad D_yv=\alpha \dd A. \]
Equivalently, the distributional Jacobian of $f$ is the matrix-valued Radon measure
\begin{equation}\label{eq:matrix-Df}
    Df=
    \begin{pmatrix}
        2\nu+\alpha \dd A & -\beta \dd A \\
        \beta \dd A & \alpha \dd A
    \end{pmatrix}.
\end{equation}
Let $e_1=(1,0)^T$. Since $DF=I \dd A+\epsilon Df$, taking symmetric parts in~\eqref{eq:matrix-Df} gives
\begin{equation}\label{eq:sym-lower}
    \operatorname{sym} DF=(1+\epsilon\alpha)I \dd A +2\epsilon\nu \, e_1\otimes e_1 \ge c_0 I \dd A,
\end{equation}
where we used $1+\epsilon \alpha \ge c_0$ almost everywhere and the positivity of $\nu$.

Choose a nonnegative function $\eta\in C_c^\infty(\mathbb C)$ satisfying $\supp \eta \subset B(0,1)$ and $\int_{\mathbb C} \eta \dd A=1$, and let $\eta_\delta(z)=\delta^{-2}\eta(z/\delta)$ be the standard mollifier. Define $F_\delta(z):=z+\epsilon(f*\eta_\delta)(z)$, then we have $DF_\delta=(DF)*\eta_\delta$. Convolution with nonnegative $\eta_\delta$ preserves the inequality~\eqref{eq:sym-lower}, which implies that $\operatorname{sym} DF_\delta(z)\ge c_0 I$ for all $z \in \mathbb C$.

For $z,w \in \mathbb C$, write $v=z-w$, then
\[ (F_\delta(z)-F_\delta(w)) \cdot v=\int_0^1 v^T DF_\delta(w+tv) v \dd t = \int_0^1 v^T (\operatorname{sym} DF_\delta)(w+tv)v \dd t \ge c_0|v|^2. \]
Because $f$ is continuous, $F_\delta \to F$ locally uniformly. Letting $\delta\downarrow0$ therefore gives
\eqref{eq:strong-monotonicity}.

By the Cauchy--Schwarz inequality,~\eqref{eq:strong-monotonicity} implies
\begin{equation}\label{eq:colipschitz}
    |F(z)-F(w)|\ge c_0|z-w| \qquad \mbox{for } z,w\in\mathbb C,
\end{equation}
so $F$ is injective, and its inverse on its image is continuous.

By invariance of domain, $F(\mathbb C)$ is open. Note that $F(\mathbb C)$ is also closed: if $F(z_j)$ converges in $\mathbb C$, then
\eqref{eq:colipschitz} shows that $(z_j)$ is a Cauchy sequence.
Writing $z=\lim_jz_j$ and using the continuity of $F$, we obtain $\lim_j F(z_j)=F(z) \in F(\mathbb C)$. Thus $F(\mathbb C)$ is a nonempty subset of the connected space $\mathbb C$ that is both open and closed, and hence $F(\mathbb C)=\mathbb C$. This proves that $F$ is a homeomorphism of $\mathbb C$ onto itself.

Moreover, taking $w=0$ in~\eqref{eq:colipschitz} gives $|F(z)| \ge |F(z)-F(0)|-|F(0)| \ge c_0 |z|-|F(0)|$. Hence $|F(z)| \to \infty$ as $|z| \to \infty$. Therefore, setting $F(\infty)=\infty$ extends $F$ to a homeomorphism of $\widehat{\mathbb C}$.

Since $\bar\partial F=\epsilon \nu$ vanishes on $\mathbb C \setminus \supp \nu$, by Weyl's lemma~\cite[Lemma A.6.10]{AIM09}, it follows that $F$ is holomorphic on $\mathbb C \setminus \supp \nu$. Moreover, $F$ is injective, so its complex derivative does not vanish there. Hence $F$ is conformal on $\mathbb C \setminus \supp \nu$. At infinity, in the coordinate $\zeta=1/z$, the function $G(\zeta):=\frac{1}{F(1/\zeta)}$ is holomorphic in a punctured neighborhood of 0 and satisfies $G(\zeta) \to 0$ as $\zeta \to 0$. Therefore, $G$ extends holomorphically to 0. The extension is injective near 0, and hence $G'(0) \neq 0$. Therefore $F$ is also conformal at infinity.

Finally, if $F$ were a M\"obius map, then, since $F$ fixes $\infty$, its restriction to $\mathbb C$ would be holomorphic. But this contradicts $\bar\partial F=\epsilon\nu \neq 0$. This completes the proof.
\end{proof}

The constructions for the Brownian graph and H\"older graphs naturally produce a real-valued function with a positive singular horizontal derivative and an absolutely continuous vertical derivative. The following proposition shows that the remaining estimate needed for Proposition~\ref{prop:criterion} is a boundedness condition on the mixed Riesz transform of the vertical derivative.

\begin{proposition}\label{prop:criterion+}
Let $q>2$, and let $v \in C_c(\mathbb C)$ be real-valued. Suppose that, in the sense of distributions,
\begin{equation}\label{eq:derivative-v}
    D_x v=\mu+a\dd A, \qquad D_y v=b\dd A,
\end{equation}
where $\mu$ is a nonzero finite positive Radon measure with compact support, and the functions $a \in C_c^\infty(\mathbb C)$ and $b \in L_c^q(\mathbb C)$ are real-valued. Suppose further that $\mathcal R_x\mathcal R_y b \in L^\infty(\mathbb C)$. Then there exists a continuous function $f:\mathbb C \to \mathbb C$ with $\bar \partial f = \mu$ such that, for every sufficiently small $\epsilon>0$, the map $z\mapsto z+\epsilon f(z)$ extends to a non-M\"obius homeomorphism of $\widehat{\mathbb C}$ that is conformal on $\widehat{\mathbb C} \setminus \supp\mu$.
\end{proposition}

\begin{proof}
    Define $\rho=a+ib$ and $f=2v-\mathcal C(\rho \dd A)$. Since $\rho \in L_c^q(\mathbb C)$ and $q>2$,~\cite[Theorem 4.3.13]{AIM09} implies that $\mathcal C (\rho \dd A) \in C^{0,\alpha}(\mathbb C)$, where $\alpha=1-2/q>0$. In particular, $\mathcal C(\rho \dd A)$ is continuous, so $f$ is continuous.

    The Cauchy-transform identities from Section~\ref{subsec:transforms-prelim} give $\bar\partial f=\mu$ and $\partial f=\mu+g \dd A$, where, by~\eqref{eq:beurling-riesz},
    \[ g=(a-ib)-\mathcal{B}(a+ib)=(a-ib)-(\mathcal R_y^2-\mathcal R_x^2+2i \mathcal R_x \mathcal R_y)(a+ib). \]
    The $L^q$-boundedness of $\mathcal B$ implies $g \in L^q(\mathbb C) \subset L_{\rm loc}^1(\mathbb C)$. Moreover, since $a$ and $b$ are real-valued, taking the real part gives $\operatorname{Re} g = (\operatorname{id}-\mathcal R_y^2+\mathcal R_x^2)a+ 2 \mathcal R_x \mathcal R_y b$. The first term is bounded since $a \in C_c^\infty(\mathbb C)$, and the second term is bounded by assumption. Thus $\operatorname{Re} g \in L^\infty(\mathbb C)$, and the conclusion now follows from Proposition~\ref{prop:criterion}.
\end{proof}

\begin{remark}\label{rem:identification}
    The function $f$ constructed in the proof is in fact a continuous representative of the Cauchy transform of $\mu$, although this identification was not needed.
    
    Indeed,~\eqref{eq:derivative-v} implies $2\bar\partial v=\mu+\rho \dd A$. Recall that $E(z)=1/(\pi z)$ satisfies $\bar\partial E=\delta_0$. Since $v$ is compactly supported, the standard differentiation rule for convolution of distributions gives $\mathcal C(\bar\partial v)=E*(\bar\partial v)=(\bar\partial E)*v=v$. Therefore $2v=\mathcal C \mu+\mathcal C(\rho \dd A)$, and hence $f=\mathcal C\mu$ in the sense of distributions.
\end{remark}

\section{Multi-scale structure of the Brownian graph}
\label{sec:prob-input}

Throughout this section and Section~\ref{sec:construct-measure}, unless explicitly stated otherwise, we work with a standard Brownian motion $(B_t)_{t \ge 0}$. We fix the exponent $q>2$ that will be used in the analytic construction.

The purpose of this section is to extract the pathwise structure needed in Section~\ref{sec:construct-measure}. Following~\cite{DM26}, we first organize the Brownian path into upboxes at nested spatial scales and obtain a uniform supply of descendants. We then prove a separate upper bound on the number of upsteps that can occur during a finite time interval. Finally, in Section~\ref{subsec:construct-tree} we use Lemma~\ref{lem:row}, together with a fast-duration condition, to construct the regular multi-scale tree on which the analytic argument will be based.

\subsection{Upboxes and descendants}
\label{subsec:prob-setup}

We begin with the stopping-time decomposition of the Brownian path that will define the vertices and edges of the tree. At each scale $H$, consecutive $H$-steps form a global partition of time; passing from $H$ to $H/K$ therefore gives a natural notion of descendants. We will need two facts about this decomposition: each upbox typically has many upward descendants in every relevant row, and, conditional on the sequence of step directions, the corresponding Brownian path segments are independent.

Let $\tau_1=\inf\{t>0:|B_t|=1\}$ and let $\mathbb P^+$ be the law of $(B_t)_{t \in[0,\tau_1]}$ conditional on $B_{\tau_1}=1$.

For $H>0$, define $T_{H,0}=0$ and, for each $k \ge 0$,
\begin{equation}\label{eq:global-partition}
    T_{H,k+1}=\inf\{t>T_{H,k}:|B_t-B_{T_{H,k}}|=H \}.
\end{equation}

For $H>0$, an \emph{$H$-step} refers to an interval $[T_{H,k},T_{H,k+1}]$ for some $k \ge 0$. It is called an \emph{$H$-upstep} if $B_{T_{H,k+1}}-B_{T_{H,k}}=H$. For an $H$-upstep $[T_{H,k},T_{H,k+1}]$, the rectangle $P:=[T_{H,k},T_{H,k+1}] \times [B_{T_{H,k}},B_{T_{H,k+1}}]$ is called an \emph{$H$-upbox}\footnote{Our definition of upboxes is slightly different from that in~\cite{DM26} as we do not shift the lower boundary downward.}. We use the terms \emph{upstep} and \emph{upbox} when the scale need not be specified.

For an $H$-upbox $P=[T_{H,k},T_{H,k+1}] \times [B_{T_{H,k}},B_{T_{H,k+1}}]$, we call $[T_{H,k},T_{H,k+1}]$ the time interval of the upbox and denote it by $I(P)$; we call $[B_{T_{H,k}},B_{T_{H,k+1}}]$ the \emph{height range} of the upbox and denote it by $J(P)$. By slightly abusing notation, we also refer to $J(P)$ as the height range of the corresponding upstep.
Conditional on $B_{T_{H,k}}$ and $B_{T_{H,k+1}}$ (but not on the value of $T_{H,k+1}-T_{H,k}$), the process $(B_t)_{t \in [T_{H,k},T_{H,k+1}]}$ has the law of a Brownian motion started at $B_{T_{H,k}}$ that is conditioned to hit $B_{T_{H,k}}+H$ before $B_{T_{H,k}}-H$ and is stopped once this occurs. Note that this is a translated and rescaled version of $\mathbb P^+$. Probability estimates for a single upbox below refer to this upward-exit conditioning; any additional conditions imposed during the tree construction will be stated explicitly.

Let $K$ be a positive integer. For $H>0$ and $h=H/K$, it is clear that all $H$-steps $\{[T_{H,k},T_{H,k+1}]\}_{k \ge 0}$ form a partition of $[0,\infty)$ (except that those intervals may intersect at the endpoints), and all $h$-steps $\{[T_{h,k'},T_{h,k'+1}]\}_{k' \ge 0}$ give a finer partition. We thus define the descendant relation between an $H$-upbox and an $h$-upbox according to their time intervals.
Specifically, let $P=[T_{H,k},T_{H,k+1}] \times [b,b+H]$ be an $H$-upbox, where $b=B_{T_{H,k}} \in H\mathbb Z$ and $b+H=B_{T_{H,k+1}}$.
For an integer $j$, the \emph{$K$-descendants in row $j$} of $P$ are defined as the collection of the $h$-upboxes $Q$ of the form $Q=[T_{h,k'},T_{h,k'+1}] \times [b+jh,b+(j+1)h]$ for some $k' \ge 0$, where $b+jh=B_{T_{h,k'}}$ and $b+(j+1)h=B_{T_{h,k'+1}}$, and $[T_{h,k'},T_{h,k'+1}] \subseteq [T_{H,k},T_{H,k+1}]$.

The first ingredient is a uniform lower bound on the number of descendants available in each relevant row of a given upbox. This is the form of the upcrossing estimate from~\cite{DM26} that will be used in the branching construction.

\begin{lemma}\label{lem:row}
For every $\eta \in (0,1)$ there exist a rational number $r \in (0,1/8)$, $c \in (0,1)$, and a positive integer $K_0$, which depend only on $\eta$, such that for every integer $K \ge K_0$ with $rK \in \mathbb{N}$, the following holds with $h=H/K$. For an $H$-upbox $P=[T_{H,k},T_{H,k+1}] \times [b,b+H]$, the event that
\[ P \mbox{ has at least } cK \mbox{ many } K\mbox{-descendants in row } j \mbox{ for each } j \in \mathcal J_K:=\{-3rK,-3rK+1,\ldots,K-rK-1\} \]
occurs with conditional probability at least $1-\eta$.
\end{lemma}

\begin{proof}
This is essentially~\cite[Lemma~2.2]{DM26}. More precisely, it is shown that for every $\eta \in (0,1)$ there exist a rational number $r \in (0,1/8)$, $c \in (0,1)$, and a positive integer $K_0$, which depend only on $\eta$, such that for every integer $K \ge K_0$ with $rK\in\mathbb{N}$, the following event occurs with probability at least $1-\eta$: the stopped Brownian motion $(B_t)_{t \in [0,\tau_1]}$ conditional on $B_{\tau_1}=1$ makes at least $cK$ upcrossings of every interval of the form
\[ [j/K,(j+1)/K], \quad \mbox{for } j \in \mathcal J_K=\{-3rK,-3rK+1,\ldots,K-rK-1\}. \]

Note that each upcrossing of the interval $[b+jh,b+(j+1)h]$ during the time interval $I(P)$ naturally induces (immediately before this upcrossing ends) an $h$-upbox whose height range is $[b+jh,b+(j+1)h]$, and different upcrossings correspond to different upboxes. The desired claim then follows from translation and Brownian scaling.
\end{proof}

We remind the readers that, however, since $j \in \mathcal J_K$ can be negative, these $K$-descendants (viewed as closed subsets of $\mathbb R^2$) in Lemma~\ref{lem:row} are not necessarily contained in $P$.

To iterate the preceding estimate down the tree, we also need independence between the finer-scale path segments once their directions have been specified. This follows immediately from the strong Markov property; for completeness, we record the precise conditional independence statement below.

\begin{lemma}\label{lem:conditional-independence}
Fix an integer $K \ge2$ and an $H$-upbox $P=[T_{H,k},T_{H,k+1}] \times [b,b+H]$, and set $h=H/K$. There is an index $r_0$ such that $T_{h,r_0}=T_{H,k}$, and we write $\sigma_\ell=T_{h,r_0+\ell}$ for $\ell \ge 0$. Let $N_P$ be the first index for which $\sigma_{N_P}=T_{H,k+1}$. Write $\epsilon_\ell=h^{-1}(B_{\sigma_{\ell+1}}-B_{\sigma_\ell}) \in \{-1,1\}$.

Let $\mathcal S_P$ be the sigma-algebra generated by $N_P$ and $\epsilon_0,\ldots,\epsilon_{N_P-1}$. Conditional on $\mathcal S_P$, the rescaled stopped Brownian path segments together with their durations,
\[ \left(h^{-2}(\sigma_{\ell+1}-\sigma_\ell),(h^{-1}(B_{\sigma_\ell+h^2t}-B_{\sigma_\ell}))_{0\le t\le h^{-2}(\sigma_{\ell+1}-\sigma_\ell)}\right), \quad 0\le\ell<N_P,
\]
are independent. Moreover, each rescaled Brownian path segment has law $\mathbb P^+$ or its vertical reflection depending on its sign.
\end{lemma}

\begin{proof}
We first work under the unconditioned Brownian law after $T_{H,k}$, and apply the strong Markov property successively at the stopping times $\sigma_\ell$. Before their signs are revealed, the stopped increments are independent Brownian paths run until first exit from $(-h,h)$. On the event $\{N_P=n\}$, these $n$ stopped Brownian path segments are conditionally independent given the sign sequence $(\epsilon_0,\ldots,\epsilon_{n-1})$. Since the index $N_P$ is determined by the sign sequence, conditional independent still holds given $\mathcal S_P$. After translation and Brownian scaling, conditioning an increment on its exit sign gives $\mathbb P^+$ for a positive sign and the reflected law for a negative sign.
\end{proof}

\subsection{Uniform counting of upsteps}
\label{subsec:uniform-bound}

The aim of this subsection is to prove an almost-sure upper bound on the number of $H$-upsteps, uniformly along a geometric sequence of scales, over all height ranges and all time intervals of a prescribed minimum length. This estimate will be used in Section~\ref{subsec:mixed-riesz} to control the number of level-$n$ tree upboxes that can contribute near a given point, which is the key input for the mixed Riesz transform estimate.

For $H>0$, an integer $\ell$, and a time interval $W \subseteq [0,\infty)$, let $N^+_{H,\ell}(W)$ be the number of $H$-upsteps $[T_{H,k},T_{H,k+1}]$ of $(B_t)_{t \ge 0}$ that intersect $W$ and have height range $[\ell H,(\ell+1)H]$.

\begin{lemma}\label{lem:uniform-count}
Fix $H_0>0$ and $ K \ge 2$, and set $H_n=H_0K^{-n}$. Almost surely, there exists a finite random constant $C_B<\infty$ which depends only on $H_0$, $K$, and $(B_t)_{t \ge 0}$, such that, for every $n$, every integer $\ell$, and every time interval $W \subseteq [0,1]$ of length $R \ge H_n^2$, we have
\begin{equation}\label{eq:uniform-count}
    N_{H_n,\ell}^+(W) \le C_B\sqrt{n+1} \frac{\sqrt{R}}{H_n}.
\end{equation}
\end{lemma}

\begin{proof}
The proof is divided into three steps. First, we obtain a Gaussian tail bound for a fixed triple $(H,\ell,W)$. Next, a union bound and Borel--Cantelli give the desired estimate simultaneously over the relevant height ranges and dyadic time intervals at all sufficiently small scales. Finally, we pass from dyadic intervals to arbitrary intervals and absorb the finitely many remaining scales into the random constant.

\textit{Step 1. Tail bound for a fixed interval.} Suppose $W$ has length $R$, and $N^+_{H,\ell}(W) \ge k \ge 4$ for some $H$ and $\ell$. After ignoring the at most two crossing intervals containing the endpoints of $W$, there are at least $k-2$ upsteps inside $W$. Between any two consecutive upsteps, there must be a downcrossing of the interval $[\ell H,(\ell+1)H]$. Hence, the Brownian motion makes at least $k-3$ disjoint downcrossings during $W$, and these downcrossings are independent by the strong Markov property.

Recall that for a Brownian motion $(B_t)_{t \ge 0}$, the one-sided exit time $\sigma_H:=\inf\{t>0: B_t=-H\}$ has Laplace transform $\mathbb E e^{-\lambda \sigma_H}=e^{-H\sqrt{2\lambda}}$ for $\lambda>0$ (see~\cite[Exercise 2.18]{MP10}). Let $\sigma_H^1,\ldots,\sigma_H^{k-3}$ denote i.i.d.\ copies of $\sigma_H$. Then, by Markov's inequality,
\[ \mathbb P[N^+_{H,\ell}(W) \ge k] \le \mathbb P[\sigma_H^1+\cdots+\sigma_H^{k-3} \le R] \le e^{\lambda R} \mathbb E[e^{-\lambda(\sigma_H^1+\cdots+\sigma_H^{k-3})}]=\exp(\lambda R-(k-3)H\sqrt{2\lambda}). \]
Optimizing over $\lambda>0$ gives, for $k \ge 4$,
\begin{equation}\label{eq:count-tail}
    \mathbb P[N^+_{H,\ell}(W)\ge k] \le \exp\left(-\frac{(k-3)^2H^2}{2R}\right) \le \exp\left(-\frac{k^2H^2}{32R}\right).
\end{equation}

\textit{Step 2. Proving~\eqref{eq:uniform-count} for dyadic intervals.} First, by the reflection principle,
\[ \sum_{n=1}^\infty \mathbb P\left[\sup_{0 \le t \le 5} |B_t|>n \right] \le 4\sum_{n=1}^\infty \mathbb P[B_5>n]<\infty. \]
It follows from the Borel--Cantelli lemma that, almost surely, there exists an integer $N_0 \ge 2$ such that $\sup_{0 \le t \le 5} |B_t| \le N_0$. Increasing $N_0$ if necessary, we may assume that $H_n=H_0 K^{-n} \le 1$ for all $n \ge N_0$. We now fix $n \ge N_0$. Then
\begin{equation}\label{eq:event-condition}
    \mbox{any } H_n\mbox{-upstep that intersects } [0,5] \mbox{ must have its height range contained in } [-n-2,n+2].
\end{equation}
For each dyadic number $D \in [H_n^2,2]$, consider the following collection of intervals of length $2D$,
\[ \mathcal V_D:=\{[0,2D], \ [D,3D], \ [2D,4D],\ldots\}. \]
Let $\mathsf A \ge 4$ be a large constant to be determined later. For each $V \in \mathcal V_D$, by~\eqref{eq:count-tail} we have
\begin{equation}\label{eq:count-tail-2}
    \mathbb P \left[N^+_{H_n,\ell}(V) \ge \mathsf A \sqrt{n+1} \tfrac{\sqrt{2D}}{H_n} \right] \le \exp\left(-\frac{\mathsf A^2 (n+1)}{32}\right).
\end{equation}

For fixed $n \in \mathbb N$, we now apply a union bound using~\eqref{eq:count-tail-2} over all pairs $(V,\ell)$, where
\begin{itemize}
    \item $V \in \mathcal V_D$ for some dyadic $D \in [H_n^2,2]$, such that $V$ intersects $[0,1]$, and
    \item $\ell$ is an integer satisfying $|\ell| \le 2n/H_n$. (For larger $|\ell|$, we have $N_{H_n,\ell}^+(V)=0$ by~\eqref{eq:event-condition}.)
\end{itemize}
The number of dyadic levels $D$ is at most $\log_2(2/H_n^2)+1$. For fixed dyadic $D \ge H_n^2$, the number of $V \in \mathcal V_D$ that intersect $[0,1]$ is at most $1/D+1 \le 1/H_n^2+1$. Therefore, the number of pairs $(V,\ell)$ is at most
\[ (\log_2(2/H_n^2)+1) \times (1/H_n^2+1) \times (4n/H_n+1), \]
which is $O(n^2 H_n^{-3})=O(n^2 K^{3n})$ with the implicit constant depending only on $H_0$ and $K$. Hence, by choosing $\mathsf A \ge 4$ sufficiently large, the sum of~\eqref{eq:count-tail-2} over all pairs $(V,\ell)$ decays exponentially in $n$ and is therefore summable. By the Borel--Cantelli lemma, almost surely there exists an integer $N_1 \ge N_0$ such that for every $n \ge N_1$, none of the events in~\eqref{eq:count-tail-2}, indexed by the pairs $(V,\ell)$ considered above, occurs.

\textit{Step 3. Concluding the proof.} Fix $n \ge N_1$. For each dyadic $D \in [H_n^2,2]$ and every interval $W \subseteq [0,1]$ with length $R \in [D/2,D]$, it is clear that $W$ is contained in $V \cap [0,1]$ for some $V \in \mathcal V_D$. It follows that
\[ N_{H_n,\ell}^+(W) \le N_{H_n,\ell}^+(V)< \mathsf A \sqrt{n+1} \frac{\sqrt{2D}}{H_n} \le 2 \mathsf A \sqrt{n+1} \frac{\sqrt{R}}{H_n}. \]

Finally, at the finitely many remaining levels $n<N_1$, since $H_n$ is bounded away from 0, the number of $H_n$-upsteps that intersect $[0,1]$ is finite almost surely. Thus, we can choose a finite random constant $C_B \ge 2\mathsf A$ for which the asserted estimate holds at every level. This concludes the proof.
\end{proof}

\subsection{Construction of the regular tree}
\label{subsec:construct-tree}

We now construct the regular tree on the Brownian graph $\Gamma_B$ that will support the analytic construction. Lemma~\ref{lem:row} supplies many candidate descendants in every row, while Lemma~\ref{lem:conditional-independence} allows the construction to be iterated independently down the successive branches. By retaining only descendants with short duration, we will additionally ensure a time-length bound needed later. Recall that $I(P)$ denotes the time interval of an upbox $P$ and that descendants are indexed by the rows in $\mathcal J_K$.

\begin{proposition}\label{prop:construct-tree}
Fix $q>2$ and $\eta \in (0,1/40)$. There exist a rational number $r \in (0,1/8)$, $s \in (0,1)$, $\mathsf C>0$, and a positive integer $K$ with $rK \in \mathbb N$, such that $m=\lfloor sK \rfloor>K^{1-1/q}$, and the following event occurs almost surely: there exists an infinite $m(1+2r)K$-regular tree of upboxes such that

\begin{enumerate}[(a)]
\item The root of the tree is an $H_0$-upbox $[t_1,t_2] \times [B_{t_1},B_{t_2}]$ for some interval $[t_1,t_2]=[T_{H_0,k},T_{H_0,k+1}] \subseteq [0,1]$ and a rational $H_0 \in (0,1)$; the $n$-th level of the tree consists of $H_n$-upboxes, with $H_n=H_0K^{-n}$.
\item Every $H_n$-upbox in the tree has exactly $m$ many $K$-descendants in each row $j \in \mathcal J_K$.
\item Every $H_n$-upbox $P$ in the tree satisfies $\mathrm{len}(I(P))\le \mathsf C H_n^2$.
\end{enumerate}
\end{proposition}

Before proving the proposition, we choose the parameters so that the branching recursion remains uniformly supercritical while the retained upboxes satisfy the required duration bound. Fix $q>2$ and $\eta \in (0,1/40)$. Let $r \in (0,1/8)$, $c \in (0,1)$, and $K_0 \in \mathbb N$ be the constants given by Lemma~\ref{lem:row}. Write $s=c/4$, and let $M=\lfloor cK \rfloor$, $m=\lfloor sK \rfloor$. By the Chernoff bound, there exist constants $C_0,c_0>0$, depending only on $r$ and $c$, such that for all sufficiently large $K \ge K_0$, we have
\begin{equation}\label{eq:parameter-K}
    (K+2rK) \, \mathbb P[\mathsf{Bin}(M,t)<m] \le C_0 K e^{-c_0 K} \quad \mbox{for all } t \in [0.9,1].
\end{equation}
We fix a positive integer $K \ge \max\{4,K_0\}$ with $rK \in \mathbb N$ so large that $m=\lfloor sK \rfloor>K^{1-1/q}$, and the right-hand side of~\eqref{eq:parameter-K} is at most $1/40$.

Let $\mathsf C>0$ be a constant. An $H$-upbox is called \emph{fast} if its duration is at most $\mathsf C H^2$. By translation invariance and Brownian scaling, the conditional probability that an $H$-upbox is not fast does not depend on $H$, and is equal to
\begin{equation}\label{eq:parameter-C}
    \epsilon_0:=\mathbb P^+[\mbox{the box } [0,\tau_1] \times [0,1] \mbox{ is not fast}]=\mathbb P^+[\tau_1>\mathsf C] \to 0 \quad \mbox{as} \quad \mathsf C \to \infty.
\end{equation}
We fix $\mathsf C<\infty$ so that $\epsilon_0<1/40$.

\begin{proof}[Proof of Proposition~\ref{prop:construct-tree}]

The proof proceeds in two steps. First, we use Lemma~\ref{lem:row} to show that a fixed upbox has a uniform positive probability of inducing an infinite regular tree. We then upgrade this positive-probability statement to an almost-sure one via the independence of Brownian motion.

\textbf{Step 1. Estimate for a single upbox.} As mentioned earlier, when considering an upbox $P=[T_{H,k},T_{H,k+1}] \times [B_{T_{H,k}},B_{T_{H,k+1}}]$, we work under the conditional law of $(B_t)_{t \in [T_{H,k},T_{H,k+1}]}$. Let $G(P)$ be the event considered in Lemma~\ref{lem:row}. Then $G(P)$ is measurable with respect to $\mathcal S_P$ and has conditional probability at least $1-\eta$.

For an upbox $P$, we say that $P$ is \emph{0-good} if $P$ is fast. For each $N \ge 0$, we say that $P$ is \emph{$(N+1)$-good} if the following conditions hold simultaneously:
\begin{itemize}
    \item $P$ is fast;
    \item $G(P)$ occurs, and moreover, for every $j \in \mathcal J_K$, there are at least $m$ indices $\ell \in \{1,\ldots,M\}$ for which $Q_{j,\ell}$ is $N$-good. Here, $Q_{j,1},\ldots,Q_{j,M}$ are the first $M$ elements, in chronological order according to the time interval, among the $K$-descendants of $P$ in row $j$.
\end{itemize}
A simple induction shows that if $P$ is $(N+1)$-good, then it is also $N$-good. We say that $P$ is $\infty$-good if $P$ is $N$-good for every $N \ge 0$. Note that if $P$ is $\infty$-good, then there exists an infinite regular tree rooted at $P$ that satisfies conditions (b) and (c).

By translation invariance and Brownian scaling, the conditional probability of $P$ being $N$-good does not depend on the choice of $P$. We denote this common value by $t_N=\mathbb P^+ [\mbox{the box } [0,\tau_1] \times [0,1] \mbox{ is } N\mbox{-good}]$. Lemma~\ref{lem:conditional-independence} implies that, conditional on $\mathcal S_P$, the events that $Q_{j,1},\ldots,Q_{j,M}$ are $N$-good are independent, for each fixed $j$, and each has probability $t_N$. Hence, by a union bound,
\begin{equation}\label{eq:recursion}
    t_{N+1} \ge 1-\epsilon_0-\eta -(K+2rK) \cdot \mathbb P[\mathsf{Bin}(M,t_N)<m].
\end{equation}
By the choice of $\epsilon_0$ in~\eqref{eq:parameter-C}, we have $t_0=1-\epsilon_0>0.9$. Moreover, since $\eta<1/40$ and $K$ is chosen as in~\eqref{eq:parameter-K}, an induction using~\eqref{eq:recursion} shows that $t_N \ge 0.9$ for every $N \ge 0$. Since the events that $P$ is $N$-good form a decreasing sequence, it follows that $P$ is $\infty$-good with conditional probability at least $0.9$.

\textbf{Step 2. Upgrading to an almost-sure event.} For $a>0$, recall the definition of $(T_{a,k})_{k \ge 0}$ from~\eqref{eq:global-partition}.
By the strong Markov property, the translated processes $(B_t-B_{T_{a,k}})_{t \in [T_{a,k},T_{a,k+1}]}$ for $k \ge 0$ are independent. Each segment exits upwards with probability $1/2$, and conditional on an upward exit its translated and rescaled path has law $\mathbb P^+$. Consequently, the events
\[ A_k^a:=\{B_{T_{a,k+1}}-B_{T_{a,k}}=a \ \mbox{ and } \ [T_{a,k},T_{a,k+1}] \times [B_{T_{a,k}},B_{T_{a,k+1}}] \mbox{ is } \infty\mbox{-good} \} \]
are independent and have the same probability $\mathsf p=\frac{1}{2} \mathbb P^+[\mbox{the box } [0,\tau_1] \times [0,1] \mbox{ is } \infty\mbox{-good}]>0$.

Let $E$ be the event that there exist $a \in \mathbb Q \cap(0,1)$ and $k \ge 0$ such that $A_k^a$ occurs, and $T_{a,k+1} \le 1$.

Fix any $\delta>0$. First choose a positive integer $N$ so that $(1-\mathsf p)^N<\delta/2$, and then choose one rational number $a \in (0,1)$ so small that $Na^2<\delta/2$. Since $(T_{a,k+1}-T_{a,k})_{k \ge 0}$ are i.i.d.\ with mean $a^2$, we have $\mathbb E[T_{a,N}]=Na^2$. By Markov's inequality and independence of the events $(A_k^a)_{k \ge 0}$,
\[ \mathbb P[E^c] \le \mathbb P[T_{a,N}>1]+\mathbb P[\cap_{k=0}^{N-1}(A_k^a)^c] \le Na^2+(1-\mathsf p)^N<\delta. \]
Since $\delta>0$ was arbitrary, it follows that $\mathbb P[E]=1$. On the event $E$, there exists an infinite regular tree rooted at $[T_{a,k},T_{a,k+1}] \times [B_{T_{a,k}},B_{T_{a,k+1}}]$ (with rational $H_0=a \in (0,1)$ and $[T_{a,k},T_{a,k+1}] \subseteq [0,1]$) that satisfies conditions (a), (b), and (c).
\end{proof}

\section{From the regular tree to non-removability}
\label{sec:construct-measure}

Throughout this section, we work on a fixed Brownian sample path belonging to the following almost-sure event. Proposition~\ref{prop:construct-tree} holds, and, for the corresponding value of $K$, Lemma~\ref{lem:uniform-count} holds simultaneously for every rational $H_0 \in (0,1)$. The latter follows by taking a countable intersection.

Fix one of the regular trees given by Proposition~\ref{prop:construct-tree}, with root $\mathsf R=[t_1,t_2] \times [B_{t_1},B_{t_2}]$ of size $H_0$, and write $H_n=H_0 K^{-n}$. For $n \ge 0$, denote by $\mathscr T_n$ the collection of $H_n$-upboxes in this tree, and set $\mathscr T=\cup_{n \ge 0} \mathscr T_n$. From now on, the parameters $q$, $r$, $s$, $K$, $m=\lfloor sK \rfloor$, $H_0$, the constant $\mathsf C$ in Condition (c) of Proposition~\ref{prop:construct-tree}, and the random constant $C_B$ from Lemma~\ref{lem:uniform-count} are all fixed, as are the Brownian path and the chosen tree. In particular, whenever $P \in \mathscr T_n$, its descendants will always mean the $K$-descendants of $P$ that belong to $\mathscr T_{n+1}$. If $Q \in \mathscr T_{n+1}$ is a descendant of $P \in \mathscr T_n$, we also say $P$ is the parent of $Q$. Unless otherwise stated, constants in the estimates below may depend on the fixed quantities above and on the smooth functions $\phi$ and $\omega$ chosen below, but not on the generation or on the points and radii involved.

Recall that the parameters were chosen so that
\begin{equation}\label{eq:parameter-conditions}
    m=\lfloor sK \rfloor>K^{1-1/q}>\sqrt{K}.
\end{equation}

We now convert the tree into a deterministic analytic construction on the subgraph $\Gamma_0:=\{(t,B_t): t \in [t_1,t_2]\}$. At each level, mass associated with a parent upbox will be redistributed among its selected descendants using smooth weights. The resulting sequence converges to a continuous function whose horizontal derivative is a nonzero positive measure supported on $\Gamma_0$, while the vertical derivative, after subtracting the initial term, has an $L^q$ density. In Section~\ref{subsec:mixed-riesz} we combine the structure of the vertical derivative with Lemma~\ref{lem:uniform-count} to obtain a mixed Riesz transform bound required in the application of Proposition~\ref{prop:criterion+}.

\subsection{The singular measure and vertical derivative}
\label{subsec:singular-measure}

We begin the construction by defining a system of smooth \emph{vertical} weights on the tree. The weights are chosen so that the mass of a parent is distributed equally among its selected descendants within each row and, at the same time, the contributions from adjacent rows form a smooth partition of unity. This conservation property makes it possible to produce a positive limiting horizontal derivative.

We first choose a nonnegative function $\phi\in C_c^\infty(\mathbb R)$, flat at the endpoints of its support, such that
\begin{equation}\label{eq:phi}
    \supp \phi \subseteq [-3r,1-2r], \qquad \phi=1 \ \mbox{on } [-2r,1-3r], \qquad \phi(t-1)+\phi(t)=1 \quad \mbox{for } 0 \le t \le 1.
\end{equation}
Such a function can be constructed by choosing a smooth transition on $[-3r,-2r]$, and then defining the matching transition on $[1-3r,1-2r]$ by the last identity.

For $n \ge 0$ and an upbox $P \in \mathscr T_n$ with height range $[\ell H_n,(\ell+1)H_n]$, define
\begin{equation}\label{eq:def-phi}
    \phi_{P,j}(y):=\phi(y/H_{n+1}-(\ell K+j)), \qquad \mbox{for} \ j \in \mathcal J_K=\{-3rK,\ldots,K-rK-1\}.
\end{equation}
We also write $J(P;j)=\ell H_n+j H_{n+1}+H_{n+1} [-3r,1-2r]$, an interval of length $(1+r)H_{n+1}$, so that $\supp \phi_{P,j} \subseteq J(P;j)$.

Suppose that the root $\mathsf R$ has height range $[\ell_0 H_0,(\ell_0+1)H_0]$. Choose a nonzero, nonnegative $\omega \in C_c^\infty(\mathbb R)$ satisfying $\supp \omega \subseteq \ell_0 H_0+H_1[-3rK-2r,K-rK-3r]$, and set $\theta_{\mathsf R}:=\omega$. We then define the functions $\theta_P$ recursively down the tree. More precisely, let $P \in \mathscr T_n$, and let $Q \in \mathscr T_{n+1}$ be one of the $m$ descendants of $P$ in row $j$, we define
\begin{equation}\label{eq:theta-child}
    \theta_Q(y)=\frac{1}{m} \theta_P(y)\phi_{P,j}(y).
\end{equation}
In particular, $\supp \theta_Q \subseteq \supp \theta_P \cap \supp \phi_{P,j} \subseteq \supp \theta_P$. Note that this support inclusion holds even though $Q$, viewed as a closed subset of $\mathbb R^2$, is not necessarily contained in $P$.

The recursive definition has two basic consequences that will be used throughout the construction: each weight remains localized near the height range of its upbox, and the total weight is preserved when passing from a parent to its descendants. We record these properties precisely in the following lemma.

\begin{lemma}\label{lem:conservation}
    The following properties hold.
    \begin{enumerate}[(a)]
        \item Let $n \ge 1$, and let $Q \in \mathscr T_n$ have height range $[\ell H_n,(\ell+1)H_n]$, then $\supp \theta_Q \subseteq \ell H_n+H_n[-3r,1-2r]$.
        \item For any $P \in \mathscr T$, we have $\sum_{j\in \mathcal J_K} \phi_{P,j}(y)=1$ for every $y \in \supp \theta_P$.
        \item For any $P \in \mathscr T$, we have $\sum_Q \theta_Q(y)=\theta_P(y)$, where the sum is taken over all $m(1+2r)K$ descendants $Q$ of $P$. Consequently, $\sum_{P \in \mathscr T_n} \theta_P(y)=\omega(y)$ for every $n \ge 0$.
    \end{enumerate} 
\end{lemma}

\begin{proof}
    (a) Let $P' \in \mathscr T_{n-1}$ be the parent of $Q$, and suppose that $P'$ has height range $[\ell'H_{n-1},(\ell'+1)H_{n-1}]$. If $Q$ is a descendant of $P'$ in row $j$, then its height range is $[(\ell'K+j)H_n, (\ell'K+j+1)H_n]$, so that $\ell=\ell'K+j$. By~\eqref{eq:theta-child}, we have
    \[ \supp \theta_Q \subseteq \supp \phi_{P',j} \subseteq J(P';j)=\ell H_n+H_n[-3r,1-2r]. \]
    This proves (a).
    
    (b) Let $P \in \mathscr T_n$ have height range $J(P)=[\ell H_n,(\ell+1)H_n]$, and define the auxiliary interval $\widetilde J(P):=\ell H_n+H_{n+1}[-3rK-2r,K-rK-3r]$. We first claim that
    \begin{equation}\label{eq:sum-phi}
        \sum_{j\in \mathcal J_K} \phi_{P,j}(y)=1 \qquad \mbox{for} \ y \in \widetilde J(P).
    \end{equation}
    For this argument, extend the notation $\phi_{P,j}$ to all $j \in \mathbb Z$ using the same formula as in~\eqref{eq:def-phi}. The last identity in~\eqref{eq:phi}, together with integer translation, gives $\sum_{j \in \mathbb Z} \phi(t-j)=1$ for every $t \in \mathbb R$. Therefore, $\sum_{j \in \mathbb Z} \phi_{P,j}(y)=1$ for every $y \in \mathbb R$.
    
    It is easy to verify that, if $y \in \widetilde J(P)$ and $j \in \mathbb Z \setminus \mathcal J_K$, then $y/H_{n+1}-(\ell K+j) \notin (-3r,1-2r)$. Since $\phi$ can be nonzero only in $(-3r,1-2r)$, it follows that $\phi_{P,j}(y)=0$ whenever $j \notin \mathcal J_K$. This proves~\eqref{eq:sum-phi}.

    We now show that $\supp \theta_P \subseteq \widetilde J(P)$. For $P=\mathsf R$, this follows directly from the choice of $\omega=\theta_{\mathsf R}$. Suppose instead that $P \neq \mathsf R$. By Part (a),
    \[ \supp \theta_P \subseteq \ell H_n+H_n[-3r,1-2r] = \ell H_n+H_{n+1}[-3rK,K-2rK] \subseteq \widetilde J(P), \]
    where we used $K \ge 4$ in the final step. Combining this inclusion with~\eqref{eq:sum-phi} proves (b).

    (c) Fix $P \in \mathscr T$. For each $j \in \mathcal J_K$, the upbox $P$ has exactly $m$ descendants in row $j$. Summing~\eqref{eq:theta-child} over all descendants $Q$ of $P$, we obtain
    \[ \sum_Q \theta_Q(y)=\sum_{j \in \mathcal J_K} m \cdot \frac{1}{m} \theta_P(y) \phi_{P,j}(y)=\theta_P(y) \sum_{j \in \mathcal J_K} \phi_{P,j}(y). \]
    If $y \in \supp \theta_P$, Part (b) shows that the last sum equals 1, and therefore $\sum_Q \theta_Q(y)=\theta_P(y)$. If $y \notin \supp \theta_P$, then both sides vanish, so the same identity holds for every $y \in \mathbb R$.
    
    Finally, summing this identity over all $P \in \mathscr T_n$ gives $\sum_{Q \in \mathscr T_{n+1}} \theta_Q(y)=\sum_{P \in \mathscr T_n} \theta_P(y)$. A simple induction on $n$ yields $\sum_{P \in \mathscr T_n} \theta_P(y)=\theta_{\mathsf R}(y)=\omega(y)$ for every $n \ge 0$. This completes the proof.
\end{proof}

We will also need quantitative control of the variation of these weights. Each passage to the next generation contributes a factor $1/m$ and introduces oscillation only at the new scale $H_{n+1}$; differentiating the recursive formula therefore gives the following uniform estimates.

\begin{lemma}\label{lem:theta-derivatives}
For every integer $k \ge 0$, there is a constant $C_k<\infty$ such that, for every $n \ge 0$, every upbox $P \in \mathscr T_n$, and every $j \in \mathcal J_K$,
\begin{equation}\label{eq:theta-phi-derivative-bound}
    \|(\theta_P \phi_{P,j})^{(k)}\|_{L^\infty} \le C_k m^{-n} H_{n+1}^{-k}.
\end{equation}
\end{lemma}

\begin{proof}
We first show that for some constant $C_k'<\infty$,
\begin{equation}\label{eq:theta-derivative-bound}
    \|\theta_P^{(k)}\|_{L^\infty} \le C_k' m^{-n} H_n^{-k}.
\end{equation}

Let $P_0=\mathsf R, P_1,\ldots,P_{n-1}, P_n=P$ be the chain of ancestors of $P$, and suppose that $P_{i+1}$ is a descendant of $P_i$ in row $j_i$. Iterating~\eqref{eq:theta-child}, we obtain
\[ \theta_P(y)=m^{-n} \cdot \omega(y) \prod_{i=0}^{n-1} \phi_{P_i,j_i}(y). \]
For every integer $a \ge 0$, the definition~\eqref{eq:def-phi} of $\phi_{P_i,j_i}$ gives $\| \phi_{P_i,j_i}^{(a)} \|_{L^\infty} \le c_a H_{i+1}^{-a}$, where $c_a$ depends only on $a$ and $\phi$. When the above product is differentiated $k$ times, each derivative on a factor $\phi_{P_i,j_i}$ contributes a factor bounded by a constant multiple of $H_{i+1}^{-1}$, while the derivatives of $\omega$ are uniformly bounded. It follows that $\|\theta_P^{(k)}\|_{L^\infty} \le C m^{-n} (1+\sum_{i=1}^n 1/H_i)^k$. This proves~\eqref{eq:theta-derivative-bound} since $1+\sum_{i=1}^n 1/H_i \lesssim 1/H_n$.

Note that $(\theta_P \phi_{P,j})^{(k)}=\sum_{i=0}^k \binom{k}{i} \theta_P^{(i)} \phi_{P,j}^{(k-i)}$. Since $\| \phi_{P,j}^{(a)} \|_{L^\infty} \le c_a H_{n+1}^{-a}$, the estimate~\eqref{eq:theta-derivative-bound} and $H_n \ge H_{n+1}$ imply~\eqref{eq:theta-phi-derivative-bound}. This proves the lemma.
\end{proof}

We next introduce the \emph{horizontal} part of the construction. To each time interval we associate a smooth probability density and its cumulative distribution function. This allows the mass carried by a parent interval to be redistributed smoothly among the time intervals of its descendants.

For every time interval $I$, choose a normalized density $p_I\in C_c^\infty(\mathbb R)$ which is positive in $\operatorname{int}(I)$, flat at the endpoints, and satisfies $\supp p_I=I$ and $\int_{\mathbb R} p_I(x) \dd x=1$. Write $F_I(x)=\int_{-\infty}^x p_I(t) \dd t$.

Given $P \in \mathscr T_n$ and $j \in \mathcal J_K$, denote the $m$ descendants of $P$ in row $j$ by $P_{j,1},\ldots,P_{j,m}$. Define
\begin{equation}\label{eq:G-p-j}
    G_{P,j}(x)=\frac{1}{m} \sum_{\ell=1}^m F_{I(P_{j,\ell})}(x)-F_{I(P)}(x).
\end{equation}
Then $G_{P,j}$ is supported in $I(P)$. Since $0 \le F_I \le 1$, we also have $|G_{P,j}(x)| \le 1$.

For each $n \ge 0$, define
\begin{equation}\label{eq:def-un}
    u_n(x,y)=\sum_{P \in \mathscr T_n}\theta_P(y)F_{I(P)}(x).
\end{equation}

\begin{lemma}\label{lem:measure-construction}
    The sequence $(u_n)_{n \ge 0}$ converges uniformly to a continuous function $u$, with $u-u_0$ continuous and compactly supported. Moreover, $\mu:=D_x u$ is a nonzero finite positive Radon measure such that $\supp \mu \subseteq \Gamma_0$.
\end{lemma}

\begin{proof}
We first prove the uniform convergence. Substituting~\eqref{eq:theta-child} into~\eqref{eq:def-un} yields
\begin{align}\label{eq:un-increment}
    u_{n+1}(x,y)-u_n(x,y) &= \sum_{P \in \mathscr T_n} \sum_{j \in \mathcal J_K} \sum_{\ell=1}^m \theta_{P_{j,\ell}}(y) F_{I(P_{j,\ell})}(x) - \sum_{P \in \mathscr T_n} \theta_P(y)F_{I(P)}(x) \nonumber \\
    &= \sum_{P \in \mathscr T_n} \sum_{j \in \mathcal J_K} \sum_{\ell=1}^m \frac{1}{m} \theta_P(y) \phi_{P,j}(y) F_{I(P_{j,\ell})}(x) - \sum_{P \in \mathscr T_n} \sum_{j \in \mathcal J_K} \theta_P(y) \phi_{P,j}(y) F_{I(P)}(x) \nonumber \\
    &=\sum_{P \in \mathscr T_n}\sum_{j \in \mathcal J_K} \theta_P(y) \phi_{P,j}(y) G_{P,j}(x),
\end{align}
where we used Lemma~\ref{lem:conservation}(b) in the second equality, and~\eqref{eq:G-p-j} in the last equality.

The function $G_{P,j}$ vanishes outside $\operatorname{int}(I(P))$. Since the intervals $I(P)$, $P \in \mathscr T_n$, have pairwise disjoint interiors, for each fixed $x$ at most one $P \in \mathscr T_n$ contributes nontrivially to~\eqref{eq:un-increment}. Once $P$ is fixed, $\supp \phi_{P,j} \subseteq J(P;j)$ implies that, for each fixed $y$, at most two indices $j \in \mathcal J_K$ contribute. Since $|G_{P,j}(x)| \le 1$, it follows from Lemma~\ref{lem:theta-derivatives} (with $k=0$) that $\|u_{n+1}-u_n\|_{L^\infty} \le 2C_0 m^{-n}$. Thus $u_n$ converges uniformly to a continuous function $u$. Since each $u_{n+1}-u_n$ is supported in the same compact set $I(\mathsf R) \times \supp \omega$, we obtain that $u-u_0 \in C_c(\mathbb C)$.

    We next consider the horizontal derivative. Since $F_I'=p_I$, differentiating~\eqref{eq:def-un} with respect to $x$ yields
\begin{equation}\label{eq:Dxun}
    D_xu_n =\sum_{P \in \mathscr T_n} \theta_P(y) p_{I(P)}(x) \dd A.
\end{equation}
In particular, each $D_x u_n$ is a positive Radon measure.
By Lemma~\ref{lem:conservation}(c), the measures in~\eqref{eq:Dxun} have a common finite total mass $\int_{\mathbb R} \omega(y) \dd y>0$. Moreover, $\supp D_x u_n \subseteq I(\mathsf R) \times \supp \omega$ for every $n$. Uniform convergence of $u_n$ implies $D_x u_n \to D_x u$ in the sense of distributions, so $\mu:=D_x u$ is a positive Radon measure supported on the same compact set. Since the total mass of $\mu$ is finite and positive, $\mu$ is finite and nonzero.

It remains to determine the support of $\mu$. For $n \ge 1$, let $A_n=\bigcup_{P \in \mathscr T_n} I(P) \times \supp \theta_P$. Each $A_n$ is compact. If $Q$ is a descendant of $P$, then $I(Q) \subset I(P)$ and $\supp \theta_Q \subset \supp \theta_P$, and therefore $A_{n+1} \subseteq A_n$. By~\eqref{eq:Dxun}, every $D_xu_N$ with $N\ge n$ is supported in $A_n$, and hence $\supp \mu \subset \cap_n A_n$.

Now take $(t,y) \in \cap_n A_n$. For each $n$, choose $P_n \in \mathscr T_n$ such that $(t,y) \in I(P_n) \times \supp \theta_{P_n}$ and let $t_n$ denote the left endpoint of $I(P_n)$. By Condition (c) of Proposition~\ref{prop:construct-tree}, $|t-t_n|\le \mathsf C H_n^2$. On the other hand, Lemma~\ref{lem:conservation}(a) gives $|y-B_{t_n}| \le \max\{3r,1-2r\} H_n \le H_n$. Since $H_n \to 0$, we have $t_n \to t$ and $B_{t_n} \to y$, and continuity of the Brownian path gives $y=B_t$. Thus $(t,y) \in \Gamma_0$, which proves $\supp \mu \subset \Gamma_0$. This completes the proof.
\end{proof}

It remains to control the vertical derivative. Unlike the horizontal derivative, no singular part should appear in this direction. Using the derivative estimates from Lemma~\ref{lem:theta-derivatives}, we obtain an absolutely continuous vertical derivative with finite $L^q$-norm.

\begin{lemma}\label{lem:vertical-derivative}
    There exists a function $b \in L_c^q(\mathbb C)$ such that $D_y (u-u_0)=b \dd A$.
\end{lemma}

\begin{proof}
For $N \ge 1$, we write $u_N-u_0=\sum_{n=0}^{N-1} (u_{n+1}-u_n)$. Differentiating~\eqref{eq:un-increment} with respect to $y$ gives $D_y (u_N-u_0)=(\sum_{n=0}^{N-1} \mathsf b_n) \dd A$, where, for $n \ge 0$,
\begin{equation}\label{eq:b-n}
    \mathsf b_n(x,y) := \partial_y(u_{n+1}-u_n) =\sum_{P \in \mathscr T_n} \sum_{j \in \mathcal J_K} G_{P,j}(x) \,(\theta_P\phi_{P,j})'(y).
\end{equation}

We now bound $\mathsf b_n$ in $L^q$. As in the proof of Lemma~\ref{lem:measure-construction}, for each fixed $(x,y)$, at most one $P \in \mathscr T_n$ and at most two indices $j \in \mathcal J_K$ contribute nontrivially to~\eqref{eq:b-n}. Since $|G_{P,j}(x)| \le 1$, Lemma~\ref{lem:theta-derivatives} (with $k=1$) yields $|\mathsf b_n(x,y)| \le 2C_1 m^{-n} H_{n+1}^{-1}$.

We next estimate the area of $\supp \mathsf b_n$. Since $G_{P,j}$ vanishes outside $\operatorname{int}(I(P))$, the projection of $\supp \mathsf b_n$ onto the $x$-axis is contained in $\cup_{P \in \mathscr T_n} I(P) \subset I(\mathsf R) \subset [0,1]$. For fixed $x$, at most one $P \in \mathscr T_n$ contributes nontrivially to~\eqref{eq:b-n}. For this $P$, all possible $y$-values lie in $\cup_{j \in \mathcal J_K} \supp \phi_{P,j}$. Since $|\mathcal J_K|=(1+2r)K$ and $\supp \phi_{P,j} \subseteq J(P;j)$ with length $(1+r)H_{n+1}$, the length of the union is bounded by $(1+2r)K \cdot (1+r) H_{n+1} \le 2H_n$. Hence the area of $\supp \mathsf b_n$ is at most $2H_n$.

Combining the pointwise bound with the support estimate, we obtain
\[ \|\mathsf b_n\|_{L^q}^q=\int_{\mathbb R^2} |\mathsf b_n(x,y)|^q \dd A \le 2H_n (2C_1 m^{-n} H_{n+1}^{-1})^q \lesssim (K^{q-1}/m^q)^n. \]
By the parameter choice~\eqref{eq:parameter-conditions}, we have $K^{q-1}/m^q<1$. Hence $\sum_{n=0}^\infty \|\mathsf b_n\|_{L^q}<\infty$. It follows that the series $\sum_{n=0}^\infty \mathsf b_n$ converges in $L^q(\mathbb C)$; denote its limit by $b:=\sum_{n=0}^\infty \mathsf b_n$. Moreover, each $\mathsf b_n$ is supported in the same compact set $I(\mathsf R) \times \supp \omega$. Therefore $b$ is also supported in this compact set, and $b \in L_c^q(\mathbb C)$. Finally, $\sum_{n=0}^{N-1} \mathsf b_n \to b$ in $L^q(\mathbb C)$, and therefore in the sense of distributions. On the other hand, the uniform convergence $u_N \to u$ implies $D_y(u_N-u_0) \to D_y(u-u_0)$ distributionally. By uniqueness of distributional limits, $D_y(u-u_0)=b \dd A$. This completes the proof.
\end{proof}

\subsection{Mixed Riesz transform bound and Proof of Theorem~\ref{thm:construction}}
\label{subsec:mixed-riesz}

The preceding subsection gives $b \in L_c^q(\mathbb C)$, and the usual $L^q$-boundedness of the Riesz transforms therefore yields only $\mathcal T b \in L^q(\mathbb C)$, where $\mathcal T=\mathcal R_x \mathcal R_y$. For the application of Proposition~\ref{prop:criterion+}, however, we need an $L^\infty$ bound. We obtain this stronger estimate by combining a pointwise bound for the individual terms in $b$ with the uniform upstep estimate of Lemma~\ref{lem:uniform-count}.

\begin{proposition}\label{prop:riesz-bounded}
    The series $\sum_{n \ge 0} \mathcal T \mathsf b_n$ converges uniformly on $\mathbb C$ to $\mathcal Tb$. In particular, $\mathcal T b \in L^\infty(\mathbb C)$.
\end{proposition}

We first establish an elementary pointwise estimate for a single term in the expression~\eqref{eq:b-n}.

\begin{lemma}\label{lem:single}
Let $a_1 \ge 1$, $a_2>0$, $h>0$, and $L>0$. Let $I,J \subset \mathbb R$ be bounded intervals with $\mathrm{len}(I)=L$ and $\mathrm{len}(J) \le a_1h$. Suppose that
\begin{enumerate}[(i)]
    \item $G \in C_c^\infty(\mathbb R)$, $\supp G \subseteq I$, and $\|G\|_{L^\infty} \le 1$.
    \item $\gamma \in C_c^\infty(\mathbb R)$, $\supp \gamma \subseteq J$, $\|\gamma\|_{L^\infty} \le a_2/h$, and $\|\gamma'\|_{L^\infty} \le a_2/h^2$.
\end{enumerate}
Write $(G \otimes \gamma)(x,y)=G(x) \gamma(y)$. Then, for every $z=(x,y) \in \mathbb R^2$,
\[ |\mathcal T (G \otimes \gamma)(z)| \le 2 a_1 a_2 L \, (h+\dist(z,I\times J))^{-2}. \]
\end{lemma}

\begin{proof}
Fix $z=(x,y) \in \mathbb R^2$, and set $d=\dist(z,I \times J)$. We split into two cases.

\textbf{Case 1.} $d \ge h$. Let $d_x=\dist(x,I)$ and $d_y=\dist(y,J)$, then we have $\max\{d_x,d_y\} \ge d/\sqrt{2}$.

Suppose first that $d_y \ge d/\sqrt{2}$. Note that the kernel satisfies $|K_{\mathcal T}(X,Y)|=\frac{1}{\pi} \frac{|XY|}{(X^2+Y^2)^2} \le 1/(4Y^2)$. Since $\|G\|_{L^\infty} \le 1$ and $\|\gamma\|_{L^\infty} \le a_2/h$, we have
\begin{align*}
    |\mathcal T(G \otimes \gamma)(x,y)| &\le \iint_{I \times J} |K_{\mathcal T}(x-s,y-t)| \, |G(s)| \, |\gamma(t)| \dd s \dd t \\
    &\le \frac{1}{4} d_y^{-2} \iint_{I \times J} |G(s)| \, |\gamma(t)| \dd s \dd t \le \frac{1}{2} d^{-2} \cdot (L \cdot a_1 h) \cdot \frac{a_2}{h} \\
    &=\frac{1}{2} a_1 a_2 L d^{-2} \le 2 a_1 a_2 L \, (d+h)^{-2}.
\end{align*}
The case $d_x \ge d/\sqrt{2}$ can be handled similarly, using instead $|K_{\mathcal T}(X,Y)| \le 1/(4X^2)$.

\textbf{Case 2.} $d \le h$. For $\delta>0$, let $\mathcal T_\delta$ denote the truncation of $\mathcal T$ with kernel $K_{\mathcal T_\delta}(X,Y)=K_{\mathcal T}(X,Y) \mathbf{1}_{X^2+Y^2>\delta^2}$. Then
\[ \mathcal T_\delta (G \otimes \gamma)(x,y)=\iint_{s^2+t^2>\delta^2} K_{\mathcal T}(s,t) G(x-s) \gamma(y-t) \dd s \dd t. \]
For fixed $s $, the set $\{t : s^2+t^2>\delta^2\}$ is symmetric with respect to $t$, whereas $K_{\mathcal T}(s,t)$ is odd in $t$. Thus $\int_{s^2+t^2>\delta^2} K_{\mathcal T}(s,t) \dd t=0$ and, consequently,
\[ \mathcal T_\delta (G \otimes \gamma)(x,y) = \int_{\mathbb R} G(x-s) \int_{s^2+t^2>\delta^2} K_{\mathcal T}(s,t) (\gamma(y-t)-\gamma(y)) \dd t \dd s. \]
By the mean value theorem, we have $|\gamma(y-t)-\gamma(y)| \le \|\gamma'\|_{L^\infty} |t| \le (a_2/h^2) |t|$. Hence, for every $s$, the absolute value of the inner integral is bounded by
\[ \frac{a_2}{h^2} \int_{s^2+t^2>\delta^2} |K_{\mathcal T}(s,t)|\, |t|\dd t \le \frac{a_2}{h^2} \int_{\mathbb R} \frac{1}{\pi} \frac{|s|t^2}{(s^2+t^2)^2} \dd t \le \frac{a_2}{2h^2}. \]
Since $\|G\|_{L^\infty} \le 1$ and $\supp G \subseteq I$ with $\mathrm{len}(I)=L$, integrating with respect to $s$ gives $|\mathcal T_\delta (G \otimes \gamma)(x,y)| \le a_2 L/(2h^2)$. Since this estimate is uniform in $\delta$, letting $\delta \downarrow 0$ and applying the dominated convergence theorem yields $|\mathcal T (G \otimes \gamma)(x,y)| \le a_2 L/(2h^2)$, which does not exceed $2a_1a_2 L (d+h)^{-2}$. This completes the proof.
\end{proof}

We also include the following refinement which uses cancellation in the horizontal variable. It gives an additional decay factor and will be used for the case of H\"older graphs in Section~\ref{subsec:holder-riesz}.

\begin{lemma}\label{lem:single+}
Under the hypotheses of Lemma~\ref{lem:single}, suppose in addition that $\int_{\mathbb R} G(s) \dd s=0$. Then, for every $z=(x,y) \in \mathbb R^2$ with $\dist(z, I \times J)>0$, we have
\[ |\mathcal T (G \otimes \gamma)(z)| \le C a_1 a_2 L^2 \dist(z, I \times J)^{-1} (h+\dist(z,I\times J))^{-2}, \]
where $C>0$ is a universal constant.
\end{lemma}

\begin{proof}
    Define $\widetilde G(s)=\frac{1}{L} \int_{-\infty}^s G(t) \dd t$. Since $G$ is supported in $I$ and has integral zero, $\widetilde G \in C_c^\infty(\mathbb R)$ is also supported in $I$. Moreover, $\|\widetilde G\|_{L^\infty}\le1$ and $L \widetilde G'=G$. Applying Lemma~\ref{lem:single} to $\widetilde G$ and $\gamma$ gives
    \begin{equation}\label{eq:intermediate-step}
        |\mathcal T(\widetilde G \otimes \gamma)(v)| \le 2a_1a_2L (h+\dist(v,I\times J))^{-2} \quad \mbox{for all } v \in \mathbb C.
    \end{equation}
    Since $\mathcal T$ commutes with distributional derivatives, we have $\mathcal T(G \otimes \gamma)=L \partial_x \mathcal T(\widetilde G \otimes \gamma)$. We then estimate this derivative using harmonicity. The kernel $K_{\mathcal T}$ is harmonic on $\mathbb C \setminus \{0\}$, so $\mathcal T(\widetilde G \otimes \gamma)$ is harmonic on $\mathbb C \setminus (I \times J)$. On this open set, the kernel integral is absolutely convergent and may be differentiated locally under the integral sign.

    Fix $z$ and write $d=\dist(z,I \times J)>0$. Then $B(z,d/2)$ is disjoint from $I \times J$, and $\dist(w,I \times J) \ge d/2$ for every $w \in B(z,d/2)$. The standard gradient estimate for harmonic functions yields
    \[ |\mathcal T(G \otimes \gamma)(z)| = L|\partial_x \mathcal T(\widetilde G \otimes \gamma)(z)| \le \frac{CL}{d}\sup_{w \in B(z,d/2)}|\mathcal T(\widetilde G \otimes \gamma)(w)| \le \frac{C a_1a_2L^2}{d(h+d)^2}, \]
    after adjusting the universal constant $C>0$. This proves the lemma.
\end{proof}

We now sum the estimate of Lemma~\ref{lem:single} over the $n$-th level of the tree. The fast-duration condition controls the horizontal size of each rectangle, while Lemma~\ref{lem:uniform-count} controls how many such rectangles can occur near a given point.

\begin{proof}[Proof of Proposition~\ref{prop:riesz-bounded}]
Recall from~\eqref{eq:b-n} that each $\mathsf b_n$ is smooth and compactly supported, so $\mathcal T \mathsf b_n \in L^\infty(\mathbb C)$. We claim that there exists a constant $C<\infty$ such that for all sufficiently large $n$,
\begin{equation}\label{eq:T-b-bounded}
    \|\mathcal T \mathsf b_n\|_{L^\infty} \le C \sqrt{n+1} (\sqrt{K}/m)^n.
\end{equation}
Since $\sqrt{K}/m<1$ by~\eqref{eq:parameter-conditions}, this estimate implies $\sum_{n=0}^\infty \|\mathcal T \mathsf b_n\|_{L^\infty}<\infty$, and hence $\sum_{n \ge 0} \mathcal T \mathsf b_n$ converges uniformly to a bounded function. On the other hand, by the $L^q$-convergence of $\sum_{n \ge 0} \mathsf b_n$ to $b$ (Lemma~\ref{lem:vertical-derivative}) and the boundedness of $\mathcal T$ on $L^q(\mathbb C)$, the same series converges to $\mathcal T b$ in $L^q(\mathbb C)$. Therefore its uniform limit agrees with $\mathcal T b$ almost everywhere. Identifying $\mathcal T b$ with this bounded representative, we conclude that $\mathcal T b \in L^\infty(\mathbb C)$.

Fix $n$ sufficiently large and $z \in \mathbb C$. For $P \in \mathscr T_n$ and $j \in \mathcal J_K$, we have $\|G_{P,j}\|_{L^\infty} \le 1$, $\supp G_{P,j} \subseteq I(P)$, and $\mathrm{len}(I(P)) \le \mathsf CH_n^2$ by Condition (c) of Proposition~\ref{prop:construct-tree}. Moreover, $\supp (\theta_P \phi_{P,j}) \subseteq \supp \phi_{P,j} \subseteq J(P;j)$ and $J(P;j)$ has length $(1+r)H_{n+1}$. By Lemma~\ref{lem:theta-derivatives}, the derivative bounds in Lemma~\ref{lem:single} hold for $\gamma=(\theta_P \phi_{P,j})'$ with $a_2=\max\{C_1,C_2\} m^{-n}$ and $h=H_{n+1}$. Applying Lemma~\ref{lem:single} with $a_1=1+r$ gives
\[ \left| \mathcal T (G_{P,j} \otimes (\theta_P \phi_{P,j})')(z) \right| \lesssim m^{-n} H_n^2 \cdot (H_{n+1}+\dist(z,I(P) \times J(P;j)))^{-2}. \]
Summing over all $P \in \mathscr T_n$ and $j \in \mathcal J_K$, we obtain
\[ |\mathcal T \mathsf b_n(z)| \lesssim m^{-n} H_n^2 \sum_{P \in \mathscr T_n} \sum_{j \in \mathcal J_K} (H_{n+1}+\dist(z,I(P) \times J(P;j)))^{-2}. \]

For $P \in \mathscr T_n$ and $j \in \mathcal J_K$, let $z_{P,j}$ be the center of the rectangle $I(P) \times J(P;j)$. Since $\mathrm{len}(I(P)) \le \mathsf C H_n^2$ and $J(P;j)$ has length $(1+r)H_{n+1}$, we have, for all sufficiently large $n$, that the diameter of this rectangle is at most $\mathsf CH_n^2+(1+r)H_{n+1} \le 2H_{n+1}$. It follows that
\[ |z-z_{P,j}| \le \dist(z,I(P) \times J(P;j))+H_{n+1}. \]

Define the discrete measure $\Lambda_n =\sum_{P \in \mathscr T_n} \sum_{j \in \mathcal J_K} \delta_{z_{P,j}}$, where $\delta_z$ is the Dirac measure at $z$. Then
\[ |\mathcal T \mathsf b_n(z)| \lesssim m^{-n} H_n^2 \sum_{P \in \mathscr T_n} \sum_{j \in \mathcal J_K} (H_{n+1}+|z-z_{P,j}|)^{-2}= m^{-n} H_n^2 \int_{\mathbb C} (H_{n+1}+|z-w|)^{-2} \dd \Lambda_n(w). \]

We now estimate the mass growth of $\Lambda_n$. Let $R>H_{n+1}$. If $z_{P,j} \in B(z,R)$, then $I(P) \times J(P;j)$ intersects $B(z,R)$. We therefore count the rectangles $I(P) \times J(P;j)$ that can intersect this ball. Recall that if $P$ has height range $[\ell H_n,(\ell+1)H_n]$, then $J(P;j) = (\ell K+j)H_{n+1}+H_{n+1}[-3r,1-2r]$. Note that the vertical intervals $\ell' H_{n+1}+H_{n+1}[-3r,1-2r]$ are translates of each other by integer multiples of $H_{n+1}$. Thus the projection of $B(z,R)$ onto the $y$-axis can intersect at most $(2R/H_{n+1})+4$ such intervals. Since $r \in (0,1/8)$, for each integer $\ell'$, there are at most two pairs $(\ell,j) \in \mathbb Z \times \mathcal J_K$ such that $\ell'=\ell K+j$.

Fix such a pair $(\ell,j)$. For sufficiently large $n$, we have $R>H_{n+1}>H_n^2$. If $R<1/2$, by Lemma~\ref{lem:uniform-count} and extending the interval if necessary, we derive that the number of $P \in \mathscr T_n$ with height range $[\ell H_n,(\ell+1)H_n]$ such that $I(P)$ intersects the $x$-projection of $B(z,R)$ is at most $C_B \sqrt{n+1} \frac{\sqrt{2R}}{H_n}$. If $R \ge 1/2$, the same upper bound also holds by applying Lemma~\ref{lem:uniform-count} to $W=[0,1]$.

Combining the vertical and horizontal enumerations gives
\begin{equation}\label{eq:Lambda_mass}
    \Lambda_n(B(z,R)) \le \left(\frac{2R}{H_{n+1}}+4 \right) \times 2 \times C_B \sqrt{n+1} \frac{\sqrt{2R}}{H_n} \lesssim C_B \sqrt{n+1} \frac{R^{3/2}}{H_n H_{n+1}}.
\end{equation}

For $j \ge 1$, define the dyadic annulus $A_j=\{w \in \mathbb C: 2^j H_{n+1} \le |w-z| < 2^{j+1} H_{n+1}\}$. Applying~\eqref{eq:Lambda_mass} to each ball $B(z,2^j H_{n+1})$, $j \ge 1$, yields
\begin{align*}
    &\quad \int_{\mathbb C} (H_{n+1}+|z-w|)^{-2} \dd \Lambda_n(w) \\
    &=\int_{B(z,2H_{n+1})} (H_{n+1}+|z-w|)^{-2} \dd \Lambda_n(w) + \sum_{j=1}^\infty \int_{A_j} (H_{n+1}+|z-w|)^{-2} \dd \Lambda_n(w) \\
    &\le H_{n+1}^{-2} \cdot\Lambda_n(B(z,2H_{n+1}))+\sum_{j=1}^\infty (2^j H_{n+1})^{-2} \cdot \Lambda_n(B(z,2^{j+1} H_{n+1})) \\
    &\lesssim C_B \sqrt{n+1} (H_n H_{n+1})^{-1} \left( H_{n+1}^{-2} \cdot (2H_{n+1})^{3/2} + \sum_{j=1}^\infty (2^j H_{n+1})^{-2} (2^{j+1} H_{n+1})^{3/2} \right) \\
    &\lesssim C_B \sqrt{n+1} H_n^{-1} H_{n+1}^{-3/2}.
\end{align*}
Substituting this into the preceding estimate for $\mathcal T \mathsf b_n$, we find
\[ |\mathcal T \mathsf b_n(z)| \lesssim m^{-n} H_n^2 \times C_B \sqrt{n+1} H_n^{-1} H_{n+1}^{-3/2} \lesssim \sqrt{n+1} (\sqrt{K}/m)^n. \]
Since the bound is uniform in $z$, this proves~\eqref{eq:T-b-bounded} and thereby completes the proof.
\end{proof}

\begin{proof}[Proof of Theorem~\ref{thm:construction}]
We continue to work on the almost-sure event fixed at the beginning of Section~\ref{sec:construct-measure}. By Lemmas~\ref{lem:measure-construction} and~\ref{lem:vertical-derivative}, the real-valued function $v=u-u_0$ belongs to $C_c(\mathbb C)$, and satisfies $D_x v=\mu+a \dd A$ and $D_y v=b \dd A$, where $\mu$ is a nonzero finite positive Radon measure with $\supp \mu \subseteq \Gamma_0 \subseteq \Gamma_B$, $a(x,y):=-\partial_x u_0(x,y)=-\omega(y) p_{I(\mathsf R)}(x) \in C_c^\infty(\mathbb C)$, and $b \in L_c^q(\mathbb C)$. Proposition~\ref{prop:riesz-bounded} gives $\mathcal R_x \mathcal R_y b \in L^\infty(\mathbb C)$. Hence Theorem~\ref{thm:construction} follows readily from Proposition~\ref{prop:criterion+}.
\end{proof}

\section{Non-removable H\"older graphs}\label{sec:holder-graph}

In this section, we give a deterministic implementation of the analytic criterion to construct non-removable $\alpha$-H\"older graphs for every $\alpha \in (0,2/3)$, thereby proving Theorem~\ref{thm:holder-graph}. The mass-redistribution argument is parallel to that of Section~\ref{sec:construct-measure}, but reaching H\"older exponents beyond those available from Brownian motion requires additional geometric input. Instead of relying on the Brownian counting estimate, we prescribe the horizontal scales of the retained intervals and carefully select descendants corresponding to both upcrossings and downcrossings. These descendants are arranged so that, in each row, their average midpoint coincides with the midpoint of the parent interval. We also choose each horizontal density $p_I$ to be symmetric about the midpoint of $I$. Together, these choices ensure that $G_{P,j}$ has integral zero (Lemma~\ref{lem:G-p-j-sum-zero}), allowing us to apply the refined single-term estimate of Lemma~\ref{lem:single+}.

We first construct the H\"older continuous function and its tree of retained intervals, then define the singular measure and vertical derivative, and finally prove the mixed Riesz transform bound. All parameters and auxiliary functions used in this section are chosen independently of those in Sections~\ref{sec:prob-input} and~\ref{sec:construct-measure}. We nevertheless deliberately reuse the notation to emphasize the common structure of the two constructions. All unspecified constants in this section are deterministic and may depend on the parameters and smooth functions chosen here.

\subsection{Construction of the H\"older continuous function}
\label{subsec:holder-construction}

Throughout the section, we fix $0<\alpha<2/3$. Choose $q>2$ such that $\alpha<q/(2q-1)$, and then choose $1-1/q<\beta<\min\{1,1/\alpha-1\}$.

Let $r=1/10$, and choose a sufficiently large positive integer $K$ such that $rK \in \mathbb N$. Write $M=\lceil K^\beta \rceil$, $N=(1+2r)K$, $T=(1+r)K$, and $S=2T+2MN$. Then we have
\begin{equation}\label{eq:parameter-choice}
    M>K^{1-1/q}, \qquad S^\alpha<K, \qquad K^2<2MS.
\end{equation}
Indeed, the second inequality follows from $\alpha(1+\beta)<1$, and the third inequality follows from $2MS>4M^2N>K^2$. We fix these parameters for the remainder of the section. The three inequalities in~\eqref{eq:parameter-choice} serve different purposes: the first gives the $L^q$-summability of the vertical derivative, the second ensures $\alpha$-H\"older continuity of the limiting graph, and the third gives the summability of the mixed Riesz transform estimates.

We now construct the desired function by an iterative replacement procedure. The basic building block is a piecewise linear function which joins its two endpoint values while making many excursions across a slightly enlarged vertical interval. At each stage, the function on every interval retained from the previous stage is replaced by an affine rescaling of this building block. We then retain certain subintervals from its oscillatory part and group them into rows according to their vertical ranges. Iterating this procedure yields a regular tree of intervals.

Define the piecewise linear function $v:[0,1] \to \mathbb R$ as follows:
\begin{equation*}
    v(x)=\begin{cases}
        \frac{S}{K}x, & \mbox{for } x \in [0,\frac{T}{S}], \\
        (1+r)-\frac{S}{K}(x-\frac{T+2mN}{S}), & \mbox{for } x \in [\frac{T+2mN}{S},\frac{T+(2m+1)N}{S}], \mbox{ where } 0 \le m \le M-1, \\
        -r+\frac{S}{K}(x-\frac{T+(2m+1)N}{S}), & \mbox{for } x \in [\frac{T+(2m+1)N}{S},\frac{T+(2m+2)N}{S}], \mbox{ where } 0 \le m \le M-1, \\
        (1+r)-\frac{S}{K}(x-\frac{T+2MN}{S}), & \mbox{for }  x \in [\frac{T+2MN}{S},\frac{T+2MN+rK}{S}], \\
        1, & \mbox{for } x \in [\frac{T+2MN+rK}{S},1].
    \end{cases}
\end{equation*}
In other words, starting from height 0, the graph first rises to height $1+r$, and makes $M$ complete oscillations between $1+r$ and $-r$, then descends to 1 and remains constant. On each linear piece, its slope is one of $\pm S/K$ and 0.

Define $v_+,v_-:[0,1] \to \mathbb R$ by $v_+=v$ and $v_-=1-v$. Then $v_+$ joins 0 to 1, whereas $v_-$ joins 1 to 0, and in either case $v_\pm([0,1])=[-r,1+r]$.

For $\bullet \in \{+,-\}$ and $j \in \widetilde{\mathcal J}_K:=\{-rK,-rK+1,\ldots,(1+r)K-1\}$, let $\mathcal D_{\bullet,j}$ denote the collection of closed subintervals $I$ of the oscillatory part $[\frac{T}{S},\frac{T+2MN}{S}]$ such that $v_\bullet|_I$ is linear and maps $I$ bijectively onto $[\frac{j}{K},\frac{j+1}{K}]$. The collection $\mathcal D_{\bullet,j}$ consists of exactly $2M$ intervals, corresponding to $M$ upward crossings and $M$ downward crossings. Every such interval has length $1/S$.

For $n \ge 0$, let $\mathscr B_n=\{[\frac{k-1}{S^n},\frac{k}{S^n}] : 1 \le k \le S^n\}$ be the collection of $S$-adic intervals of level $n$.

We define recursively a sequence of continuous piecewise linear functions $f_n$ and collections of intervals $\mathscr A_n \subseteq \mathscr B_n$. Start with $f_0(x)=x$, $P_0=[0,1]$, and $\mathscr A_0=\{P_0\}$. Suppose that $f_n:[0,1] \to \mathbb R$ and $\mathscr A_n$ have been defined so that, for every $P=[a,b] \in \mathscr A_n$, the function $f_n$ is linear on $P$ and $|f_n(b)-f_n(a)|=K^{-n}$. Let $y_P=\min\{f_n(a),f_n(b)\}$ and $\bullet_P \in \{+,-\}$ be the sign of $f_n(b)-f_n(a)$. On $P$, define $f_{n+1}(x)=y_P+K^{-n} v_{\bullet_P}(S^n(x-a))$. Outside $\bigcup_{P \in \mathscr A_n} P$, we set $f_{n+1}=f_n$. Since $v_+$ joins 0 to 1 and $v_-$ joins 1 to 0, this replacement preserves the endpoint values of $f_n$ on every $P\in\mathscr A_n$. In particular, $f_{n+1}$ is continuous.

For $P=[a,b] \in \mathscr A_n$ and $j \in \widetilde{\mathcal J}_K$, the descendants of $P$ in row $j$ are the $2M$ intervals $a+S^{-n} D$, where $D \in \mathcal D_{\bullet_P,j}$. We let $\mathscr A_{n+1}$ be the collection of all such descendants, as $P$ ranges over $\mathscr A_n$ and $j$ ranges over $\widetilde{\mathcal J}_K$. Every descendant has length $S^{-(n+1)}$ and belongs to $\mathscr B_{n+1}$. Moreover, if $Q$ is a descendant of $P$ in row $j$, then the values of $f_{n+1}$ at the endpoints of $Q$ are $y_P+j K^{-(n+1)}$ and $y_P+(j+1) K^{-(n+1)}$. In particular, $f_{n+1}$ is linear on each descendant and its values at the endpoints differ by $K^{-(n+1)}$ in absolute value. Thus every $P \in \mathscr A_n$ has exactly $2M$ descendants in each of the $N$ rows, and hence exactly $2MN$ descendants in total. Iterating the construction produces a $2MN$-regular tree\footnote{We remind the reader that the vertices of this tree are encoded by intervals, rather than by upboxes as in Section~\ref{sec:construct-measure}.} $\mathscr A=\bigcup_{n \ge 0} \mathscr A_n$ of intervals.

The choice of descendants has the following simple property, which will later imply the zero-integral identity in Lemma~\ref{lem:G-p-j-sum-zero}.

\begin{lemma}\label{lem:holder-center}
Let $P \in \mathscr A_n$ and $j \in \widetilde{\mathcal J}_K$. Denote the $2M$ descendants of $P$ in row $j$ by $P_{j,1},\ldots,P_{j,2M}$. Then $\frac{1}{2M} \sum_{\ell=1}^{2M} c(P_{j,\ell})=c(P)$, where $c(I)$ denotes the midpoint of an interval $I$.
\end{lemma}

\begin{proof}
By an affine rescaling, it suffices to prove the assertion for $P=[0,1]$. For either $v_+$ or $v_-$, the oscillatory interval is $[\frac{T}{S},\frac{T+2MN}{S}]$, whose midpoint is $1/2$. The result follows by symmetry.
\end{proof}

We next show the convergence and H\"older regularity of the resulting function.

\begin{lemma}\label{lem:holder-regularity}
The sequence $(f_n)_{n \ge 0}$ converges uniformly to an $\alpha$-H\"older continuous function $f:[0,1] \to \mathbb R$.
\end{lemma}

\begin{proof}
Since $v([0,1])=[-r,1+r]$, replacing $f_n$ by $f_{n+1}$ on an interval $P \in \mathscr A_n$ changes the function by at most $(1+r)K^{-n}$. Hence $\|f_{n+1}-f_n\|_{L^\infty([0,1])} \le (1+r)K^{-n}$, and thus $f_n$ converges uniformly to a continuous function $f:[0,1]\to\mathbb R$.

We next prove the H\"older estimate. We first claim that the oscillation of $f$ on every $S$-adic interval $P \in \mathscr B_n$ is at most $5 K^{-n}$.

Suppose first that $P \in \mathscr A_n$, then the oscillation of $f$ on $P$ is at most $K^{-n}+2 \sum_{j \ge n} (1+r)K^{-j} \le 5K^{-n}$. If instead $P \notin \mathscr A_n$, then for some $j \le n-1$ we have that $f=f_{j+1}$ on $P$. On this interval, $f_{j+1}$ is linear with slope of absolute value at most $(S/K)^{j+1}$, hence the oscillation is at most $(S/K)^{j+1} \cdot S^{-n} \le K^{-n}<5K^{-n}$.

Now, for $0 \le x<y \le 1$, choose $n \ge 0$ so that $S^{-(n+1)} \le |x-y| \le S^{-n}$. Then the interval $[x,y]$ is contained in the union of at most two adjacent $S$-adic intervals in $\mathscr B_n$, and consequently
\[ |f(x)-f(y)| \le 10K^{-n} \le 10S^{-\alpha n} \le 10 S^\alpha \cdot |x-y|^\alpha, \]
where we used~\eqref{eq:parameter-choice}. This shows that $f$ is $\alpha$-H\"older continuous.
\end{proof}

\subsection{The singular measure and vertical derivative}
\label{subsec:construct-measure-holder}

Fix the function $f$ as constructed in Section~\ref{subsec:holder-construction}. For $n \ge 0$, write $L_n=S^{-n}$ and $H_n=K^{-n}$. By construction, if $P \in \mathscr A_n$, then every subsequent replacement preserves the values at the endpoints of $P$. Hence these values agree with those of the limiting function $f$. In particular, the values of $f$ at the endpoints of $P$ are $y_P$ and $y_P+H_n$. Moreover, if $Q$ is a descendant of $P$ in row $j \in \widetilde{\mathcal J}_K$, then $y_Q=y_P+j H_{n+1}$.

We now adapt the mass-redistribution construction of Section~\ref{subsec:singular-measure} to the deterministic tree $\mathscr A$. As before, the vertical weights enforce mass conservation, whereas the horizontal densities are chosen symmetric about the midpoints of the corresponding intervals.

Choose $0<\delta<r/10$ and a nonnegative function $\eta \in C_c^\infty(\mathbb R)$ with $\supp \eta \subseteq [-\delta,\delta]$ and $\int_{\mathbb R} \eta(x) \dd x=1$. Define $\phi=\mathbf 1_{[0,1]}*\eta$, then
\begin{equation}\label{eq:def-phi-holder}
    \supp \phi \subseteq [-\delta,1+\delta], \qquad 0 \le \phi \le 1, \qquad \sum_{j \in \mathbb Z} \phi(t-j)=1 \quad \mbox{for } t \in \mathbb R.
\end{equation}

For $n \ge 0$ and an interval $P \in \mathscr A_n$, define
\begin{equation}\label{eq:holder-phi-Pj}
    \phi_{P,j}(y)=\phi((y-y_P)/H_{n+1}-j), \qquad \mbox{for } j \in \widetilde{\mathcal J}_K=\{-rK,-rK+1,\ldots,(1+r)K-1\}.
\end{equation}
We also write $J(P;j)=y_P+j H_{n+1}+H_{n+1} [-\delta,1+\delta]$, an interval of length $(1+2\delta) H_{n+1}$, so that $\supp \phi_{P,j} \subseteq J(P;j)$.

Choose a nonzero, nonnegative $\omega \in C_c^\infty(\mathbb R)$ with $\supp \omega \subseteq (-\delta,1+\delta)$, and set $\theta_{P_0}=\omega$. We define the functions $\theta_P$ recursively down the tree. If $P \in \mathscr A_n$ and $Q$ is one of the $2M$ descendants of $P$ in row $j$, we define
\begin{equation}\label{eq:theta-child-holder}
    \theta_Q(y) = \frac{1}{2M} \theta_P(y)\phi_{P,j}(y).
\end{equation}

\begin{lemma}\label{lem:conservation-holder}
    The following properties hold.
    \begin{enumerate}[(a)]
        \item Let $n \ge 1$ and $Q \in \mathscr A_n$, then $\supp \theta_Q \subseteq y_Q+H_n[-\delta,1+\delta]$.
        \item For any $P \in \mathscr A$, we have $\sum_{j\in \widetilde{\mathcal J}_K} \phi_{P,j}(y)=1$ for every $y \in \supp \theta_P$.
        \item For any $P \in \mathscr A$, we have $\sum_Q \theta_Q(y)=\theta_P(y)$, where the sum is taken over all $2MN$ descendants $Q$ of $P$. Consequently, $\sum_{P \in \mathscr A_n} \theta_P(y)=\omega(y)$ for every $n \ge 0$.
    \end{enumerate} 
\end{lemma}

\begin{proof}
    The proof is almost identical to that of Lemma~\ref{lem:conservation}. For (b), one only needs to verify $\supp \theta_P \subseteq y_P+H_{n+1}[-rK+\delta,(1+r)K-\delta]$, which follows since $K$ is sufficiently large and $\delta<r/10$.
\end{proof}

The following lemma is proved in the same way as Lemma~\ref{lem:theta-derivatives}.

\begin{lemma}\label{lem:theta-derivatives-holder}
For every integer $k \ge 0$, there is a constant $C_k<\infty$ such that, for every $n \ge 0$, every interval $P \in \mathscr A_n$, and every $j \in \widetilde{\mathcal J}_K$,
\[ \|(\theta_P \phi_{P,j})^{(k)}\|_{L^\infty} \le C_k (2M)^{-n} H_{n+1}^{-k}. \]
\end{lemma}

Fix an even, nonnegative $p \in C_c^\infty((-1/2,1/2))$ with $\int_{\mathbb R} p(x) \dd x=1$. For every time interval $I$ of length $L$ and midpoint $c(I)$, define $p_I(x)=L^{-1} p(\frac{x-c(I)}{L})$ and write $F_I(x)=\int_{-\infty}^x p_I(t) \dd t$. Then $p_I$ is a smooth probability density supported in $I$ and symmetric about $c(I)$.

Given $P \in \mathscr A_n$ and $j \in \widetilde{\mathcal J}_K$, denote the $2M$ descendants of $P$ in row $j$ by $P_{j,1},\ldots,P_{j,2M}$. Define
\begin{equation}\label{eq:G-p-j-holder}
    G_{P,j}(x)=\frac{1}{2M} \sum_{\ell=1}^{2M} F_{P_{j,\ell}}(x)-F_P(x).
\end{equation}
Then $G_{P,j}$ is supported in $P$, and $|G_{P,j}(x)| \le 1$. The symmetry of $p_I$, together with the identity of Lemma~\ref{lem:holder-center}, gives the additional property that will be needed below.

\begin{lemma}\label{lem:G-p-j-sum-zero}
    For each $P \in \mathscr A_n$ and $j \in \widetilde{\mathcal J}_K$, we have $\int_{\mathbb R} G_{P,j}(x) \dd x=0$.
\end{lemma}

\begin{proof}
    For every time interval $I$, since $p_I(x)$ is symmetric with respect to $c(I)$, we have $\int_{\mathbb R} x \,p_I(x) \dd x=c(I)$. It follows from integration by parts that, for two intervals $I_1$ and $I_2$,
    \[ \int_{\mathbb R} (F_{I_1}(x)-F_{I_2}(x)) \dd x=\int_{\mathbb R} x \, (p_{I_2}(x)-p_{I_1}(x)) \dd x=c(I_2)-c(I_1), \]
    and the assertion follows from~\eqref{eq:G-p-j-holder} and Lemma~\ref{lem:holder-center}.
\end{proof}

For each $n \ge 0$, define
\begin{equation}\label{eq:def-un-holder}
    u_n(x,y)=\sum_{P \in \mathscr A_n} \theta_P(y) F_P(x).
\end{equation}

\begin{lemma}\label{lem:measure-construction-holder}
    The sequence $(u_n)_{n \ge 0}$ converges uniformly to a continuous function $u$, with $u-u_0$ continuous and compactly supported. Moreover, $\mu:=D_x u$ is a nonzero finite positive Radon measure such that $\supp \mu \subseteq \Gamma_f$.
\end{lemma}

\begin{proof}
    With Lemmas~\ref{lem:conservation-holder} and~\ref{lem:theta-derivatives-holder} in place of Lemmas~\ref{lem:conservation} and~\ref{lem:theta-derivatives}, the proof of Lemma~\ref{lem:measure-construction} applies without change and gives the uniform convergence of $u_n$, the compact support of $u-u_0$, and the fact that $\mu=D_x u$ is a nonzero finite positive Radon measure supported in $\cap_n A_n$, where $A_n=\bigcup_{P \in \mathscr A_n} P \times \supp \theta_P$. It suffices to show that $\cap_n A_n \subset \Gamma_f$. Take $(t,y) \in \cap_n A_n$. For each $n$, choose $P_n \in \mathscr A_n$ such that $(t,y) \in P_n \times \supp \theta_{P_n}$ and let $t_n$ denote an endpoint of $P_n$ such that $f(t_n)=y_{P_n}$. Since $|t-t_n| \le L_n$ and $|y-f(t_n)| \le (1+\delta)H_n$ by Lemma~\ref{lem:conservation-holder}(a), it follows from the continuity of $f$ that $y=f(t)$. This completes the proof.
\end{proof}

The vertical derivative satisfies the same estimates as in Lemma~\ref{lem:vertical-derivative}, with $m$ replaced by $2M$.

\begin{lemma}\label{lem:vertical-derivative-holder}
    There exists a function $b \in L_c^q(\mathbb C)$ such that $D_y (u-u_0)=b \dd A$.
\end{lemma}

\begin{proof}
    As in Lemma~\ref{lem:vertical-derivative}, we may take $b=\sum_{n=0}^\infty \mathsf b_n$, where
    \begin{equation}\label{eq:b-n-holder}
        \mathsf b_n(x,y) = \sum_{P \in \mathscr A_n} \sum_{j \in \widetilde{\mathcal J}_K} G_{P,j}(x)(\theta_P \phi_{P,j})'(y).
    \end{equation}
    Since $K^{1-1/q}/(2M)<1$ by~\eqref{eq:parameter-choice}, the proof follows from the same argument as in Lemma~\ref{lem:vertical-derivative}.
\end{proof}

\subsection{Mixed Riesz transform bound and Proof of Theorem~\ref{thm:holder-graph}}
\label{subsec:holder-riesz}

It remains to prove the mixed Riesz transform bound. In contrast with the Brownian case, where the uniform counting estimate controls the number of relevant upboxes, the rectangles here are arranged more regularly yet also more densely, so a direct application of Lemma~\ref{lem:single} is insufficient. split the sum into contributions from near and far intervals, applying Lemma~\ref{lem:single} to the former and the refined estimate of Lemma~\ref{lem:single+} to the latter. The key additional input for the far contribution is the zero-integral property from Lemma~\ref{lem:G-p-j-sum-zero}.

\begin{proposition}\label{prop:riesz-bounded-holder}
    The series $\sum_{n \ge 0} \mathcal T \mathsf b_n$ converges uniformly on $\mathbb C$ to $\mathcal Tb$. In particular, $\mathcal T b \in L^\infty(\mathbb C)$.
\end{proposition}

\begin{proof}
    Recall from~\eqref{eq:b-n-holder} that each $\mathsf b_n$ is smooth and compactly supported, so $\mathcal T \mathsf b_n \in L^\infty(\mathbb C)$. We claim that there exists a constant $C<\infty$ such that for all $n \ge 0$,
\begin{equation}\label{eq:T-b-bounded-holder}
    \|\mathcal T \mathsf b_n\|_{L^\infty} \le C (n+1) (K^2/(2MS))^n.
\end{equation}
Since $K^2/(2MS)<1$ by~\eqref{eq:parameter-choice}, this estimate implies $\sum_{n=0}^\infty \|\mathcal T \mathsf b_n \|_{L^\infty}<\infty$, and hence $\sum_{n \ge 0} \mathcal T \mathsf b_n$ converges uniformly to a bounded function. On the other hand, Lemma~\ref{lem:vertical-derivative-holder}, together with the boundedness of $\mathcal T$ on $L^q(\mathbb C)$, shows that the same series converges to $\mathcal T b$ in $L^q(\mathbb C)$. Therefore its uniform limit agrees with $\mathcal T b$ almost everywhere, proving the proposition once~\eqref{eq:T-b-bounded-holder} is established.

Fix $n \ge 0$ and $z=(x,y) \in \mathbb R^2$. For $P \in \mathscr A_n$ and $j \in \widetilde{\mathcal J}_K$, write $d_P=\dist(x,P)$, $e_{P,j}=\dist(y,J(P;j))$, and $D_{P,j}=\dist(z,P \times J(P;j))$. By~\eqref{eq:G-p-j-holder}, Lemmas~\ref{lem:theta-derivatives-holder} and~\ref{lem:G-p-j-sum-zero}, the pair $(G_{P,j}, (\theta_P\phi_{P,j})')$ satisfies the hypotheses of Lemma~\ref{lem:single} with $a_1=1+2\delta$ and $a_2=\max\{C_1,C_2\} (2M)^{-n}$. Moreover, $\int_{\mathbb R} G_{P,j}=0$.

For fixed $P$, recall that the intervals $J(P;j)$ are translates of each other by multiples of $H_{n+1}$ and have length $(1+2\delta)H_{n+1}<2H_{n+1}$. Hence, for each integer $k \ge 0$, there are at most four indices $j \in \widetilde{\mathcal J}_K$ satisfying $e_{P,j} \in [kH_{n+1},(k+1)H_{n+1})$. It follows that for every nonnegative decreasing integrable function $\mathsf F:[0,\infty) \to [0,\infty)$,
\begin{equation}\label{eq:summation-bound}
    \sum_{j \in \widetilde{\mathcal J}_K} \mathsf F(e_{P,j}) \le 4 \sum_{k=0}^\infty \mathsf F(kH_{n+1}) \le 4 \left( \mathsf F(0)+\frac{1}{H_{n+1}} \int_0^\infty \mathsf F(t) \dd t \right).
\end{equation}

We split the sum~\eqref{eq:b-n-holder} according to the value of $d_P$.

\begin{itemize}
    \item \textit{Intervals with $d_P<L_n$.} Since $\mathscr A_n$ consists of $S$-adic intervals of length $L_n$, there are at most three such intervals. Lemma~\ref{lem:single} gives
    \[ |\mathcal T(G_{P,j} \otimes (\theta_P \phi_{P,j})')(z)| \lesssim (2M)^{-n} L_n (H_{n+1}+D_{P,j})^{-2} \lesssim (2M)^{-n} L_n (H_{n+1}+e_{P,j})^{-2}. \]
    Applying~\eqref{eq:summation-bound} with $\mathsf F(t)=(H_{n+1}+t)^{-2}$ gives
    \begin{equation}\label{eq:near}
        \sum_{\substack{P \in \mathscr A_n \\ d_P<L_n}} \sum_{j \in \widetilde{\mathcal J}_K} |\mathcal T(G_{P,j} \otimes (\theta_P \phi_{P,j})')(z)| \lesssim (2M)^{-n} L_n H_{n+1}^{-2}.
    \end{equation}
    
    \item \textit{Intervals with $d_P \ge L_n$.} Since $\int_{\mathbb R} G_{P,j}=0$, Lemma~\ref{lem:single+} and $D_{P,j}\ge(d_P+e_{P,j})/\sqrt{2}$ imply
    \begin{equation}\label{eq:far-2}
        |\mathcal T(G_{P,j} \otimes (\theta_P \phi_{P,j})')(z)| \lesssim (2M)^{-n} L_n^2 (d_P+e_{P,j})^{-1} (H_{n+1}+d_P+e_{P,j})^{-2}.
    \end{equation}

    For $d>0$, the function $\mathsf F_d(t)=(d+t)^{-1} (H_{n+1}+d+t)^{-2}$ is decreasing and satisfies $\int_0^\infty \mathsf F_d(t) \dd t \le d^{-1} \int_0^\infty (H_{n+1}+d+t)^{-2} \dd t=(d(H_{n+1}+d))^{-1}$. Using~\eqref{eq:summation-bound} and~\eqref{eq:far-2}, we obtain
    \begin{align}\label{eq:far-3}
        \sum_{j\in\widetilde{\mathcal J}_K} |\mathcal T(G_{P,j} \otimes (\theta_P \phi_{P,j})')(z)|
        &\lesssim (2M)^{-n} L_n^2 \left( \frac{1}{d_P(H_{n+1}+d_P)^2} + \frac{1}{H_{n+1} d_P(H_{n+1}+d_P)} \right) \nonumber \\
        &\lesssim (2M)^{-n} \frac{L_n^2}{H_{n+1} d_P (H_{n+1}+d_P)}.
    \end{align}
    For each integer $k \ge 1$, there are at most two $S$-adic intervals of length $L_n$ such that $d_P \in [kL_n,(k+1)L_n)$. Since $t \mapsto (t(H_{n+1}+t))^{-1}$ is decreasing, we have
    \begin{align*}
        \sum_{\substack{P \in \mathscr A_n \\ d_P \ge L_n}} \frac{1}{d_P(H_{n+1}+d_P)} &\le 2 \sum_{k=1}^\infty \frac{1}{kL_n(H_{n+1}+kL_n)} \le \frac{2}{L_n(H_{n+1}+L_n)}+\frac{2}{L_n} \int_{L_n}^\infty \frac{1}{t(H_{n+1}+t)} \dd t \\
        &\lesssim (L_n H_{n+1})^{-1} \left(1+\log \left(1+\frac{H_{n+1}}{L_n} \right) \right).
    \end{align*}
    Together with~\eqref{eq:far-3}, this yields
    \begin{equation}\label{eq:far}
        \sum_{\substack{P \in \mathscr A_n \\ d_P \ge L_n}} \sum_{j \in \widetilde{\mathcal J}_K} |\mathcal T(G_{P,j} \otimes (\theta_P \phi_{P,j})')(z)| \lesssim (2M)^{-n} L_n H_{n+1}^{-2} \left(1+\log \left(1+\frac{H_{n+1}}{L_n} \right) \right).
    \end{equation}
\end{itemize}

Combining~\eqref{eq:near} and~\eqref{eq:far}, and
recalling~\eqref{eq:b-n-holder}, we conclude that
\[ \|\mathcal T \mathsf b_n\|_{L^\infty} \lesssim (2M)^{-n} L_n H_{n+1}^{-2} \left(1+\log \left(1+\frac{H_{n+1}}{L_n} \right) \right). \]
Since $L_n=S^{-n}$ and $H_{n+1}=K^{-(n+1)}$, this implies~\eqref{eq:T-b-bounded-holder} and thereby completes the proof.
\end{proof}

\begin{proof}[Proof of Theorem~\ref{thm:holder-graph}]
    By Lemma~\ref{lem:holder-regularity}, the function $f$ constructed above is $\alpha$-H\"older continuous. Lemmas~\ref{lem:measure-construction-holder} and~\ref{lem:vertical-derivative-holder}, together with Proposition~\ref{prop:riesz-bounded-holder}, verify the hypotheses of Proposition~\ref{prop:criterion+} with $v=u-u_0$ and with $\mu$ supported on $\Gamma_f$. It follows that there exists a non-M\"obius homeomorphism of $\widehat{\mathbb C}$ that is conformal on $\widehat{\mathbb C} \setminus \Gamma_f$. Thus $\Gamma_f$ is not conformally removable.
\end{proof}

\bibliographystyle{alpha}
\bibliography{ref}

\end{document}